\documentclass{article}
\usepackage[latin1]{inputenc}
\usepackage{amsmath}
\usepackage{amsthm,amssymb,amsfonts}
\usepackage{graphicx}
\usepackage{verbatim}
\usepackage{color,cite}
\usepackage{latexsym,epsf}
\usepackage{latexsym}
\usepackage{lscape}
\usepackage[T1]{fontenc}
\usepackage[all,cmtip]{xy}
\usepackage{hyperref}
\usepackage{caption}
\usepackage{xspace}
\usepackage{enumitem}
\usepackage{tikz}
\usepackage{xcolor}
\usepackage{arydshln}
\usepackage{appendix}
\usetikzlibrary{patterns, decorations.pathreplacing}
\usetikzlibrary{shapes,arrows,spy,positioning}

\theoremstyle{plain}
\newtheorem{theorem}{Theorem}[section]
\newtheorem{lemma}[theorem]{Lemma}
\newtheorem{proposition}[theorem]{Proposition}
\newtheorem{corollary}[theorem]{Corollary}

\newtheorem*{definition}{Definition}
\newtheorem{example}[theorem]{Example}

\newtheorem*{remark}{Remark}

\newcommand{\nc}{\newcommand}

\nc\bB{\mathbb{B}}
\nc\bC{\mathbb{C}}
\nc\bD{\mathbb{D}}
\nc\bE{\mathbb{E}}
\nc\bF{\mathbb{F}}
\nc\bG{\mathbb{G}}
\nc\bH{\mathbb{H}}
\nc\bI{\mathbb{I}}
\nc{\bJ}{\mathbb{J}}
\nc\bK{\mathbb{K}}
\nc\bL{\mathbb{L}}
\nc\bM{\mathbb{M}}
\nc\bN{\mathbb{N}}
\nc\bO{\mathbb{O}}
\nc\bP{\mathbb{P}}
\nc\bQ{\mathbb{Q}}
\nc\bR{\mathbb{R}}
\nc\bS{\mathbb{S}}
\nc\bT{\mathbb{T}}
\nc\bU{\mathbb{U}}
\nc\bV{\mathbb{V}}
\nc\bW{\mathbb{W}}
\nc\bY{\mathbb{Y}}
\nc\bX{\mathbb{X}}
\nc\bZ{\mathbb{Z}}

\nc\cA{\mathcal{A}}
\nc\cB{\mathcal{B}}
\nc\cC{\mathcal{C}}
\nc\cD{\mathcal{D}}
\nc\cE{\mathcal{E}}
\nc\cF{\mathcal{F}}
\nc\cG{\mathcal{G}}
\nc\cH{\mathcal{H}}
\nc\cI{\mathcal{I}}
\nc{\cJ}{\mathcal{J}}
\nc\cK{\mathcal{K}}
\nc\cM{\mathcal{M}}
\nc\cN{\mathcal{N}}
\nc\cO{\mathcal{O}}
\nc\cP{\mathcal{P}}
\nc\cQ{\mathcal{Q}}
\nc\cS{\mathcal{S}}
\nc\cT{\mathcal{T}}
\nc\cU{\mathcal{U}}
\nc\cV{\mathcal{V}}
\nc\cW{\mathcal{W}}
\nc\cY{\mathcal{Y}}
\nc\cX{\mathcal{X}}
\nc\cZ{\mathcal{Z}}

\nc\fc{\mathfrak{c}}

\title{Shimura-Teichm\"uller Curves in Genus $5$}
\author{David Aulicino\thanks{D. A. was partially supported by the National Science Foundation under Award Nos. DMS - 1738381, DMS - 1600360, and PSC-CUNY Grant $\sharp$~61639-00 49.} ~and Chaya Norton}
\date{}

\begin{document}

\newcommand{\genlin}{\text{GL}_2(\mathbb{R})}
\newcommand{\splin}{\text{SL}_2(\mathbb{R})}
\newcommand{\spolin}{\text{SO}_2(\mathbb{R})}

\def\blue#1{\textcolor[rgb]{0,0,1}{#1}}

\maketitle

\begin{abstract}
We prove that there are no Shimura-Teichm\"uller curves generated by genus five translation surfaces, thereby completing the classification of Shimura-Teichm\"uller curves in general.  This was conjectured by M\"oller in his original work introducing Shimura-Teichm\"uller curves.  Moreover, the property of being a Shimura-Teichm\"uller curve is equivalent to having completely degenerate Kontsevich-Zorich spectrum.

The main new ingredient comes from the work of Hu and the second named author, which facilitates calculations of higher order terms in the period matrix with respect to plumbing coordinates.  A large computer search is implemented to exclude the remaining cases, which must be performed in a very specific way to be computationally feasible.
\end{abstract}

\tableofcontents

\section{Introduction}

Shimura-Teichm\"uller curves are Teichm\"uller curves in the moduli space of Riemann surfaces that map to Shimura curves in the moduli space of Abelian varieties.  Geometrically, they are geodesics in the moduli space of Riemann surfaces with respect to the Teichm\"uller metric that map to geodesics in the moduli space of Abelian varieties with respect to the Hodge metric.  From a dynamical systems point of view, they are Teichm\"uller curves with maximally many zero Lyapunov exponents in their Kontsevich-Zorich spectrum, i.e. they have completely degenerate Kontsevich-Zorich spectrum.

The notion of a Shimura-Teichm\"uller curve was introduced in \cite{MollerShimuraTeich}.  In \cite{ForniHand} and \cite{ForniMatheus} (followed by \cite{ForniMatheusZorichSqTiled}) two Teichm\"uller curves were discovered with the property that they have completely degenerate Kontsevich-Zorich spectrum.  The work of \cite{MollerShimuraTeich} classified all such Shimura-Teichm\"uller curves except in genus five.  The first named author \cite{AulicinoCompDegKZAIS} showed that there are at most finitely many such curves in genus five by showing that no higher dimensional $\splin$ orbit closures could have completely degenerate Lyapunov spectrum.  However, such an argument could not result in an explicit bound.  The work of \cite{ForniMatheusZorichSqTiled} showed that any Shimura-Teichm\"uller curves in genus five could not be square-tiled \emph{cyclic} covers of tori.  In this paper, we complete the classification conjectured in \cite{MollerShimuraTeich}

\begin{theorem}
\label{MainThm}
There do not exist Shimura-Teichm\"uller curves in genus five.
\end{theorem}



Combined with the results of \cite{MollerShimuraTeich} and \cite{AulicinoCompDegKZAIS}, we get the following two equivalent theorems.

\begin{theorem}
\label{FullSTClassThm}
The only Shimura-Teichm\"uller curves are the ones generated by the Eierlegende Wollmilchsau and Ornithorynque.
\end{theorem}

Theorem \ref{FullCDKZClassThm} solves \cite[Problem 1]{EskinKontsevichZorich2} as conjectured by \cite{ForniMatheusZorichSqTiled}.

\begin{theorem}
\label{FullCDKZClassThm}
The only affine submanifolds of the moduli space of translation surfaces with completely degenerate Kontsevich-Zorich spectrum are the Teichm\"uller curves generated by the Eierlegende Wollmilchsau and Ornithorynque.
\end{theorem}

The key new ingredient that allows us to reduce the number of possibilities to a finite checkable list is the work of \cite{HuNortonGenVarFormsAbDiffs}.  In \cite{MasurExt}, the boundary of Teichm\"uller space was parameterized and expansions for differentials were given in terms of pinching parameters.  When restricting to the cusp of a Teichm\"uller curve, Lemma~\ref{TCIntegralRelPinchParams} derives a formula for the pinching parameters near every node in terms of a single complex parameter.  Exact formulas for higher order terms in pinching parameter expansions for differentials in the case in which a single closed curve is pinched were derived in \cite{FayThetaFcns} and corrected in \cite{Yamada}.  However, in most cases that appear in the study of dynamics in moduli space, including every case relevant to this paper, more than one curve is pinched simultaneously.  Therefore, the exact formulas for the higher order terms in \cite{FayThetaFcns, Yamada} do not apply.  In fact, the Ahlfors-Rauch variational formula, which has been a staple of almost all work on this problem, is avoided due to the power of the results of \cite{HuNortonGenVarFormsAbDiffs}.  The work of \cite{HuNortonGenVarFormsAbDiffs} derives exact formulas for the coefficient of any term in the expansion of a differential in terms of pinching parameters for any collection of curves.  Since the problem of finding Shimura-Teichm\"uller curves is equivalent to finding period matrices whose derivative has rank one, the higher terms in the Taylor expansion of the period matrix allow us to prove that this cannot be the case for all but finitely many possibilities in genus five.  The main result in provided in Theorem \ref{AtMostTwoParts}.

In spite of the reduction from infinite to finite, the exclusion of each of the finite cases is a significant computational problem even with modern technology.  This is because the square-tiled surfaces that we are forced to consider could have up to $144$ squares.  There are on the order of $144! \approx 5.6 \times 10^{249}$ square-tiled surfaces in general because square-tiled surfaces are expressed as pairs of permutations in the group $S_n$, where $n$ is the number of squares.  One permutation describes the horizontal gluing of squares, and the other describes the vertical gluing.  To avoid working with such square-tiled surfaces in general, we introduce a coordinate system that records the identification among saddle connections (commonly called cylinder diagrams), lengths of the saddle connections, and twists of the cylinders.  This reduces the number of cases to approximately $24 \times 10^{15}$ surfaces in the largest case.  The significance of this coordinate system is that partial coordinates can be proven to be inadmissible for the problem at hand.  This allows entire families of surfaces to be checked with a single arithmetic computation, thereby rendering the problem tractable.  The pseudo-algorithm is explained in Section \ref{TwoTorPtSect}.

Whenever possible, we give theoretical arguments to exclude cases that arise.  For example, the work of \cite{EskinKontsevichZorich2}, which did not exist during the writing of \cite{MollerShimuraTeich}, facilitates the exclusion of many finite cases.  See Section \ref{CaseiiandiiiSect}.

The potential existence of Shimura-Teichm\"uller curves in genus five was a significant obstacle to overcome in the work of \cite{AulicinoCompDegKZ} and \cite{AulicinoCompDegKZAIS}.  In Section \ref{CDKZSpecSect}, we explain how to modify \cite{AulicinoCompDegKZAIS} in order to get a short proof of Theorem \ref{FullCDKZClassThm}.

\subsubsection*{Acknowledgments} The authors would like to thank Matt Bainbridge for suggesting that the results of \cite{HuNortonGenVarFormsAbDiffs} could be brought to bear on this problem and for helpful discussions at the earliest stage of this work.  They would also like to thank Anton Zorich, Martin M\"oller, Samuel Grushevsky, and Xuntao Hu for helpful discussions.  The first author would like to thank Vincent Delecroix, Giovanni Forni, and Alex Eskin for numerous conversations on this topic during the course of his graduate studies and postdoc.  Finally, the authors are grateful to the Fields Institute for hosting the ``Workshop on Dynamics and Moduli Spaces of Translation Surfaces'' where this collaboration began.

\section{Definitions and Previous Results}
\label{PreliminarySection}

\subsection{Geometry of Flat Surfaces}

\noindent \textbf{Translation Surfaces}: A \emph{translation surface} is a pair $(X, \omega)$ of a Riemann surface $X$ carrying a holomorphic $1$-form $\omega$, also known as an Abelian differential.  Given a regular closed geodesic on a translation surface, there exists a maximal homotopy class of parallel trajectories that determine a \emph{cylinder}.  Define the \emph{height} of a cylinder to be the distance between boundaries and the \emph{period}, \emph{width}, and \emph{circumference} will all refer to the same quantity.  The \emph{modulus} of a cylinder is given by $h/w$, where $h$ is its height and $w$ its width.  If every parallel trajectory on the translation surface is closed, this produces a cylinder decomposition of the surface.  The maximality of the homotopy class implies that the boundaries of a cylinder are always unions of \emph{saddle connections}.  A translation surface is called \emph{completely periodic} if it has the property that a single regular closed geodesic implies that all parallel trajectories are closed, i.e. it admits a cylinder decomposition.

\

\noindent \textbf{Strata of Translation Surfaces}:  If $(X, \omega)$ is a translation surface with genus $g \geq 2$, then $\omega$ has zeros of total order $2g-2$.  Let $\kappa$ be a partition of $2g-2$.  Consider the moduli space of all translation surfaces $\cH(\kappa)$ with zeros of order specified by $\kappa$.  As usual the moduli space is taken with equivalence to mean up to action by the mapping class group.  Often, it will be convenient to employ the common shorthand exponential notation for strata, e.g. $\cH(2, 1^4) = \cH(2,1,1,1,1)$.

\

\noindent \textbf{$\genlin$ Action}: The group $\genlin$ acts on strata $\cH(\kappa)$ by splitting the differential $\omega$ into its real and imaginary parts, and multiplying by a matrix in $\genlin$:
$$\left( \begin{array}{cc} a & b \\ c & d \\ \end{array} \right) \left( \begin{array}{c} \text{Re}(\omega) \\ \text{Im}(\omega) \\ \end{array} \right).$$
We will often restrict to the subgroup $\splin$ in order to preserve the area of the translation surface.  Let $g_t$ denote the Teichm\"uller geodesic flow on the moduli space of translation surfaces, which is given by
$$g_t = \left( \begin{array}{cc} e^t & 0 \\ 0 & e^{-t} \\ \end{array} \right).$$

\

\noindent \textbf{Veech Surfaces}: One can consider the subgroup of $\text{SL}(X, \omega) \subset \splin$ of derivatives of affine diffeomorphisms that fix the translation surface $(X, \omega)$ in $\cH(\kappa)$.  In general, this group is trivial.  However, if it forms a lattice subgroup of $\splin$, then $(X, \omega)$ is called a \emph{Veech surface}, or \emph{lattice surface}, and its orbit under $\splin$ is a \emph{Teichm\"uller curve}.  (The terminology Teichm\"uller curve for this $3$-dimensional object comes from the fact that its projection to the moduli space of Riemann surfaces is an algebraic curve.)

The celebrated Veech dichotomy states that every Veech surface is completely periodic and that the straight-line flow in the non-periodic directions is uniquely ergodic \cite{VeechTeichCurvEisen}.

\

\noindent \textbf{Square-Tiled Surfaces}: The particular class of Veech surfaces that will be considered in this paper are \emph{square-tiled surfaces}, which are degree $d$ branched covers of the torus ramified over exactly one point.  Since the torus covers itself with any degree, given a branched covering of a torus ramified over a single point, one can consider the largest intermediate torus that \emph{need not} be branched over a single point.  Following \cite{MollerShimuraTeich}, call this covering the \emph{optimal covering}
$$\pi_{opt}: (X, \omega) \rightarrow (E, \eta),$$
and denote its degree by $d_{opt}$.

\

\noindent \textbf{Degenerate Surfaces}: It was proven by \cite{MasurClosedTraj} that no $\splin$ orbit is compact.  A compactification of the moduli space of Riemann surfaces can be obtained via the Deligne-Mumford compactification.  Since this paper only concerns Teichm\"uller curves, we avoid general terminology.  The boundary of a Teichm\"uller curve consists of a finite union of points corresponding to the cusps of the algebraic curve.  The surfaces corresponding to the cusps are given by considering all cylinder decompositions of the Veech surface and taking their moduli to infinity by acting by the Teichm\"uller geodesic flow.  


A cusp of a Teichm\"uller curve in the Deligne-Mumford compactification represents a nodal Riemann surface $X'$.  Removing the nodes yields a possibly disconnected union of punctured Riemann surfaces.  The connected components of the punctured Riemann surface are called \emph{parts}.  For Teichm\"uller curves, the natural Abelian differential $\omega'$ on $X'$ is meromorphic with simple poles at the punctures and holomorphic, but not identically zero everywhere else.  In particular, if $p^+$ and $p^-$ are a pair of punctures resulting from the removal of a node, then the residues of $\omega'$ satisfy
$$\text{Res}_{p^+}(\omega') = -\text{Res}_{p^-}(\omega').$$

There is a canonical \emph{dual graph} for each nodal Riemann surface given by associating a vertex to every part and an edge to every node.

\

\noindent \textbf{Plumbing Coordinates}: This is completely general and does not rely on being a Teichm\"uller curve.  A neighborhood of a node is diffeomorphic to $\{(z, w) \in \bC^2 | zw = 0\}$.  Removing the node results in two neighborhoods that are diffeomorphic to punctured discs.  If one punctured disc lies on component $X_i$ and the other on $X_j$, then we call the two punctures $q_i^+$ and the other $q_j^-$, where $i$ and $j$ are not necessarily distinct.  This notation will be clear in the specific setting of this paper explained below.  If we consider neighborhoods of $q_i^+$ and $q_j^-$ that are diffeomorphic to unit discs with local coordinates $z_i^+$ and $z_j^-$, respectively, for each small parameter $s \in \bC$ we can consider the curves $\gamma_{z_i^+}$ given by $\{|z_i^+| = \sqrt{|s|}\}$ and $\gamma_{z_j^-}$ given by $\{|z_j^-| = \sqrt{|s|}\}$.  We glue the annuli $\{z_i^+ | \sqrt{|s|} \leq |z_i^+| < 1\}$ and $\{z_j^- | \sqrt{|s|} \leq |z_j^-| < 1\}$ via the function $I(z_i^+)=\frac{s}{z_j^-}$ that identifies $\gamma_{z_i^+}$ and $\gamma_{z_j^-}$ to a single curve $\gamma$, which we call the \emph{seam}.  The parameter $s$ is referred to interchangeably as a \emph{plumbing} or \emph{pinching parameter}.  The smooth surface resulting from removing a node and identifying along a seam with plumbing parameter $s$ is called a \emph{plumbed surface} and denoted $X^s$.

\begin{remark}
In the lemma below, each node arises from the modulus of a flat cylinder tending to infinity.
\end{remark}

\begin{lemma}
\label{TCIntegralRelPinchParams}
Given a nodal surface $(X', \omega')$ in the boundary of a cusp of a Teichm\"uller curve, there exists a natural choice of local coordinates near each node $N_j$, $1 \leq j \leq n$, such that there exists a complex parameter $s$ in a sufficiently small complex disc and a unique set of natural numbers $m_j$ depending on $s$ with the property that for each node $N_j$, the pinching parameter at that node is given by $s_j = s^{m_j}$.
\end{lemma}

\begin{remark}
The natural choice of coordinates in the lemma arises from the fact that each node of a Riemann surface at the cusp of a Teichm\"uller curve corresponds to a cylinder with large modulus.  The fact that there is a unique direction on the surface that sees all of the cylinders with a large modulus is what specifies the flat structure that is used to choose the coordinates at the nodes.
\end{remark}

\begin{proof}
Following \cite{MasurExt}, a neighborhood of the boundary of moduli space is parametrized by $3g-3$ complex coordinates $(\vec \tau, \vec s) \in \bC^{3g-3}$ such that $(0, \ldots, 0)$ corresponds to $X'$ here.  The $\vec \tau$ coordinate parametrizes $X'$ and the $\vec s \in \bC^n$ coordinate is a vector of pinching parameters for each of the nodes.  For each node $N_j$, the surface is plumbed with complex parameter $s_j$ as described above to get a point in moduli space.  

The lemma follows by choosing coordinates $z_j^{\pm}$ at each node with which to perform plumbing by using the restricted geometry of surfaces on a Teichm\"uller curve.  Precisely speaking, we will choose flat coordinates below for the cylinders that give rise to each node and use the exponential map to transform them to the coordinates $z_j^{\pm}$.  Using $\splin$, we vary the flat coordinates to parametrize an entire neighborhood of the cusp of the Teichm\"uller curve.  The result follows from the fact that the ratios of the moduli of parallel cylinders is constant in a surface on a Teichm\"uller curve.

Let $(X_{(\vec \tau, \vec s)}, \omega_{(\vec \tau, \vec s)})$ be a vertically periodic surface in a neighborhood of $(X', \omega')$ such that the vertical direction realizes the cylinders with the largest modulus.  The ratio of the moduli of any pair of parallel cylinders on a Veech surface is rational, see \cite{VorobetsPlaneStrctsBilliards} or \cite[Thm. 1.3 (iii)]{SmillieWeissCharLattice}.  Furthermore, the fact that we are on a Teichm\"uller curve implies that the ratio of moduli between any fixed pair of cylinders in the neighborhood of $(X', \omega')$ is constant.  We show for a choice of local coordinates used for plumbing at the nodes, each component of $\vec s$ is an integral power of a single complex parameter.

For the cylinder $C_j$ whose core curve converges to $N_j$ as $s_j$ goes to zero, cut it in half and identify each half with a rectangle $R_j$ with corners at $0$, $\sqrt{-1}c_j$ and $h_j/2$.  Let $\zeta_j^-$ and $\zeta_j^+$ be the flat coordinates with respect to the differential $\omega_{(\vec \tau, \vec s)}$ on each half of $C_j$.  If we calculate how the corners of $R_j$ transform under multiplication by the $2$-dimensional real parameter family of matrices
$$\left(\begin{array}{cc} t & 0 \\ u & t^{-1} \end{array} \right),$$
for $t > 0$, we see that they become the corners of a parallelogram $P_j(t,u)$ at $0$, $t^{-1}\sqrt{-1}c_j$ and $t h_j/2+u(h_j/2)\sqrt{-1}$.  By definition, the parameters $(t, u)$ parametrize every surface in a neighborhood of a cusp of a Teichm\"uller curve.

Let $\mu_j = h_j/c_j$ denote the modulus of cylinder $C_j$.  The exponential map $\exp\left(\frac{-2\pi}{c_j} \zeta_j^{\pm}\right)$ maps $P_j$ to an annulus with unit outer radius and inner radius $\exp\left(-\pi\mu_j t\right)$.  The surface $(X,\omega)$ is the result of identifying pairs of annuli by gluing along the boundary with twists $e^{\sqrt{-1}\theta_j}$, for $\theta_j \in \bR$. Hence, we choose local coordinates on one annuli of the form $e^{\sqrt{-1}\theta_j}\exp\left(\frac{-2\pi}{c_j} \zeta_j^{-}\right)$ and of the form $\exp\left(\frac{-2\pi}{c_j} \zeta_j^{+}\right)$ on the other.  The surface $(X,\omega)$ is equivalent to that of plumbing with respect to the pinching parameters,
$$s_j = e^{-2\pi \mu_jt-2\pi \sqrt{-1} u\mu_j}.$$
Defining the coordinates $z_j^- = e^{\sqrt{-1}\theta_j}\exp\left(\frac{-2\pi}{c_j} \zeta_j^{-}\right)$ and $z_j^+ = \exp\left(\frac{-2\pi}{c_j} \zeta_j^{+}\right)$, produces the surface $(X,\omega)$ via the gluing map $I(z_j^+)=\frac{s_j}{z_j^-}$.  Thus, by varying $s_j$, namely sending $t$ to $\infty$, the neighborhood of the cusp on the Teichm\"uller curve is parameterized.

We now show that $s_j$ can be described by integer powers of a complex parameter $s$, which tends to zero as we degenerate to the cusp.  Since the ratio of moduli of vertical cylinders is a rational constant in the neighborhood of a cusp of a Teichm\"uller curve, there exists $m \in \bN$ depending on $s$ such that for each vertical cylinder $C_j$, there exists $m_j \in \bN$ such that
$$\mu_j = \mu_1\frac{m_j}{m}.$$
Let
$$s(t,u) = e^{-2\pi \frac{\mu_1}{m} t-2\pi \sqrt{-1} u\frac{\mu_1}{m}} = e^{\left(-2\pi \frac{\mu_1}{m}\right)(t+\sqrt{-1} u)}.$$
By inspection, $s_j = s^{m_j}$ for all $j$.  As $s$ varies, the ratios of moduli of cylinders remain constant by construction, so the neighborhood of $(X', \omega')$ restricted to the Teichm\"uller curve is indeed parameterized by $s$.  
\end{proof}

\begin{remark}
In summary, we chose coordinates near one node $N_1$ in the proof above.  By the rationality of ratios of moduli and the fact that these ratios are constant for a Teichm\"uller curve, we were able to write the coordinates near every other node in terms of the coordinates near $N_1$.  This was done via the formulas for $z_j^{\pm}$ and consequentially, $s_j$ in the proof above.

The coordinates used for plumbing are normal coordinates for a stable differential with simple poles at each node, i.e. coordinates such that the differential is locally of the form $r\frac{dz}{z}$ without any further terms in the expansion. In general, one has a $\mathbb C^*$ choice of such local coordinates near each node. 
\end{remark}

\noindent \textbf{Kontsevich-Zorich Cocycle}: Consider the bundle of first absolute cohomology $H^1$ over a Teichm\"uller curve.  The Teichm\"uller geodesic flow on the cotangent bundle descends to a trajectory on the Teichm\"uller curve, which in turn induces a cocycle $G_t^{KZ}$ on $H^1$.  One can consider the Lyapunov exponents of this cocycle, which are well-defined almost everywhere by the Oseledet's Multiplicative Ergodic Theorem.  In fact, this is a symplectic cocycle, so after normalizing the largest exponent to one, we get the \emph{Kontsevich-Zorich spectrum} of the Teichm\"uller curve, or \emph{KZ-spectrum} for short.
$$1 > \lambda_2 \geq \cdots \geq \lambda_g \geq -\lambda_g \geq \cdots \geq -\lambda_2 > -1.$$
The fact that $1 > \lambda_2$ was proven in \cite{ForniDev, VeechGt}.  This is well-defined for all $\splin$ orbits, but the greater generality is not needed here.  It was proven for strata that all of the inequalities in the KZ-spectrum are strict, i.e. the spectrum is simple \cite{ForniDev, AvilaVianaSimp}.  However, \cite{ForniHand} found the first example of a surface in genus three that generates a Teichm\"uller curve with maximally many zero exponents, which is now called the \emph{Eierlegende Wollmilchsau} \cite{HerrlichSchmithusenEier}.  We say that the KZ-spectrum is \emph{completely degenerate} if 
$$\lambda_2 = \cdots = \lambda_g = 0.$$
Later \cite{ForniMatheus, ForniMatheusZorichSqTiled} found a second example in genus four of a square-tiled surface with completely degenerate KZ-spectrum, which is now called the \emph{Ornithorynque}.  The monodromy groups of the KZ-cocycle for these examples were studied in detail in \cite{MatheusYoccoz}, where a related infinite class of cyclic surfaces were given.  (None of the higher genus examples in that class have completely degenerate KZ-spectrum.)

\

\noindent \textbf{Shimura Curves}: The moduli space of abelian varieties inherits a metric from the homogeneous space $Sp(2g,\mathbb R)/U(g)$ called the Hodge metric. A \emph{Shimura curve} is a curve which is totally geodesic in the Hodge metric on the moduli space of Abelian varieties with a complex multiplication point.  We refer the reader to \cite[$\S$ 2]{MollerShimuraTeich} for a more complete description of Shimura curves and related background.  A Teichm\"uller curve that maps to a Shimura curve under the Torelli map is called a \emph{Shimura-Teichm\"uller curve} or \emph{ST-curve} for short.  One implication in Proposition \ref{EqToCDKZSpec} below is given by \cite[Prop. 6.4]{MollerShimuraTeich}.  The other direction follows from \cite[Cor. 7.1]{ForniHand}(see Lemma \ref{DerPerMat} below) and \cite[$\S$ 2]{MollerShimuraTeich}.

\begin{proposition}
\label{EqToCDKZSpec}
A Teichm\"uller curve is an ST-curve if and only if the KZ-spectrum is completely degenerate.
\end{proposition}

Since the property of having completely degenerate Kontsevich-Zorich spectrum will be used to explain the relevance of the variational formulas from \cite{HuNortonGenVarFormsAbDiffs}, we follow the biased convention that any reference to an ST-curve in this paper will refer to a Teichm\"uller curve in the moduli space of translation surfaces and \emph{not} a Shimura curve in the moduli space of Abelian varieties.

\

\noindent \textbf{Period Matrix}: Let $X$ be a Riemann surface of genus $g \geq 2$.  Let $\{a_1, \ldots, a_g, b_1, \ldots, b_g\}$ be a symplectic basis of $H_1(X, \bZ)$.  The basis $\{\theta_1, \ldots, \theta_g\}$ of Abelian differentials on $X$ is called a \emph{normalized basis} if it is normalized so that
$$\int_{a_i} \theta_j = \delta_{ij},$$
where $\delta_{ij}$ is the Kronecker delta.  The following lemma about a basis of holomorphic differentials will be needed in the proof of Theorem \ref{AtMostTwoParts} below.

\begin{lemma}
\label{NonZeroAtPt}
Let $X$ be a Riemann surface of genus $g \geq 1$ carrying a basis of holomorphic Abelian differentials $\{\theta_1, \ldots, \theta_g\}$.  If $p \in X$, then there exists $i$ such that $\theta_i$ does not have a zero at $p$.
\end{lemma}

\begin{proof}
By contradiction, assume that there is a point $p$ where a basis of differentials vanish. Let $K$ be the canonical line bundle.  This implies $h^0(K+p)=h^0(K)=g$, where $h^0(L)$ is the dimension of the space of holomorphic sections of a line bundle $L$ over $X$. It then follows from Riemann-Roch that $h^0(X,-p)=2$, and there is a non-constant meromorphic function on $X$ with one simple pole. This implies $X$ is the Riemann sphere.
\end{proof}

The \emph{period matrix} $\Pi(X)$ is the $g \times g$ complex matrix with components defined by
$$\tau_{ij} = \int_{b_i}\theta_j.$$
It is a classical result that it is symmetric with positive definite imaginary part.

\

\noindent \textbf{Derivative of the Period Matrix}: The space of Beltrami differentials is dual to the cotangent space of quadratic differentials on $X$.  In the moduli space $\cH(\kappa)$, one can consider a Beltrami differential $\mu$ and take a derivative of $\Pi(X)$ in $\cH(\kappa)$ in direction $\mu$.  

Given an Abelian differential $\omega$, a unique Beltrami differential is given by the quotient
$$\mu_{\omega} = \frac{\bar \omega}{\omega},$$
which is defined everywhere except at the finitely many zeros and poles of $\omega$.  

The significance of this for us is \cite[Cor. 7.1]{ForniHand}.

\begin{lemma}[Forni]
\label{DerPerMat}
An $\splin$ orbit has completely degenerate Kontsevich-Zorich spectrum if and only if the derivative of the period matrix at every point on the orbit is a rank one matrix.
\end{lemma}

\noindent \textbf{Cauchy Kernel}: Given a choice of Lagrangian subspace in $H_1(X,\mathbb Z)$, namely a choice of $a$-cycles, as well as a point $q_0$ on $X$, the \emph{Cauchy kernel}, denoted $K(z,w)$, is the unique differential in $z$ with two simple poles at $z=w$ and $z=q_0$ with residues $\pm\frac{1}{2\pi i}$ and vanishing $a$-periods. We emphasize that the Cauchy kernel is a \emph{differential} in $z$ and a multi-valued \emph{function} in $w$. The reader can verify that in all places where the Cauchy kernel is used, the expressions are well-defined and independent of the marked point $q_0$, see \cite[Sec. 3.3]{HuNortonGenVarFormsAbDiffs} for details. 

Often we will want to remove the pole and consider the holomorphic part of the Cauchy kernel. In particular, we will do this in the neighborhood of a node, and this is done in the usual way. In a fixed local coordinate chart, define
$$\bK(z(p), z(q)) = K(z(p),z(q))-\frac{dz(p)}{2\pi i(z(p)-z(q))}.$$ For convenience of notation, we will also define $\bK(z, w) = K(z,w)$ when $z$ and $w$ are disjoint coordinate neighborhoods on the Riemann surface, i.e. the Cauchy kernel is holomorphic, and there is no pole to remove.

\

\noindent \textbf{Fundamental Normalized Bi-differential}: The \emph{fundamental normalized bi-differential} is $\omega(z,w) = 2\pi id_w K(z,w)$ with the properties that 
\begin{itemize}
\item it is a differential in each variable $z$ and $w$,
\item it has a double pole along the diagonal $z = w$ with bi-residue equal to one,
\item it is symmetric in $z$ and $w$, and
\item it is normalized to have vanishing $a$-periods in each variable.
\end{itemize}
While $\omega$ will be used to refer to a $1$-form, $\omega(z,w)$ with the variables $z$ and $w$ will be exclusively reserved for a bi-differential.

The key fact concerning the fundamental normalized bi-differential which we will use in the proof of Theorem \ref{AtMostTwoParts} is that $b$-periods of the fundamental bi-differential produce the normalized basis of holomorphic differentials in the remaining variable.

\

\noindent \textbf{Jump Problem}:  The jump problem can be stated for any number of parts sharing any number of nodes by linearity. For simplicity of notation we introduce the jump problem with two parts sharing one node. Let $X_1'$ and $X_2'$ be two distinct parts sharing a node on a nodal Riemann surface.  Given differentials $\omega_1'$ and $\omega_2'$ supported on $X_1'$ and $X_2'$, respectively, and satisfying the opposite residue condition at the node (if there is a non-zero residue), we can ask for a differential on the plumbed surface $X^s$ given by plumbing the node between $X_1'$ and $X_2'$.  In spite of satisfying the opposite residue condition, there is no reason a priori to expect the differentials $\omega'_1$ and $\omega'_2$ to match along the seam.  The discrepancy across the seam is called a \emph{jump}.  The \emph{jump problem} involves finding a holomorphic differential, defined up to a normalization condition, on the plumbed surface minus the seams whose jumps across the seams are as prescribed by the problem. The classical solution is, for example, contained in \cite[Ch. 1 $\S$5: Thm. 5.1]{RodinRiemBdProbRiemSurfs}.  However, this solution is not useful for understanding variational formulas where we are interested in studying a parametric jump problem in terms of the plumbing parameters. It involves the Cauchy kernel on the plumbed surface, which itself varies non-trivially in plumbing parameters.  The second named author with Samuel Grushevsky and Igor Krichever introduced a new approach to study the parametric jump problem in~\cite{GrushevskyKricheverNortonRealNormDiffs}.  The second named author with Xuntao Hu then applied the approach to the setup relevant to this paper in~\cite{HuNortonGenVarFormsAbDiffs}.

\begin{example}
Consider the jump problem posed by taking two tori joined at a node where one torus carries a non-zero differential $\omega$ and the other carries the zero differential.  Then the jump problem would be to find a global parametrically defined differential on the plumbed genus two surface minus the seam to correct the error on the seam which arises because $\omega$ and zero do not match across the seam, and such that when the pinching parameter is sent to zero, we recover the original two differentials.
\end{example}


\subsection{Results on ST-Curves}

The following was proven in \cite[Cor. 3.3, $\S$3.6]{MollerShimuraTeich} or \cite[Lem. 4.5]{AulicinoCompDegKZAIS}.

\begin{lemma}
\label{STCurveArith}
Every surface generating an ST-curve is square-tiled.
\end{lemma}

The following was proven in \cite[Lem. 5.3]{MollerShimuraTeich} and in \cite[Cor. 5.6]{AulicinoCompDegKZ} in general.

\begin{proposition}
\label{AllCylsHom}
Given a cylinder decomposition in any periodic direction on a square-tiled surface generating an ST-curve, the core curves of the cylinders generate a $1$-dimensional subspace of homology, i.e. every cylinder is homologous to every other cylinder.
\end{proposition}

It was proven in \cite[Lem. 5.9]{AulicinoCompDegKZ} in general using the terminology connectivity graph and also in \cite[Lem. 5.2]{MollerShimuraTeich} for ST-curves that the dual graph always has the following form.  Proposition \ref{AllCylsHom} and Lemma \ref{STDualGraphCycle} are equivalent statements.

\begin{lemma}
\label{STDualGraphCycle}
The dual graph of a boundary point of an ST-curve is always a cycle, also called a ring.
\end{lemma}

For the following corollary, see \cite[Lem. 5.2, 5.4]{MollerShimuraTeich}.

\begin{corollary}
\label{PartsPosGenus}
Each part of a degenerate surface in the cusp of an ST-curve has positive genus.
\end{corollary}

\begin{proof}
The boundary between two homologous adjacent cylinders consists of zeros of order at least two.  Pinching the core curves of those cylinders results in exactly two simple poles.  Therefore, by Riemann-Roch the genus of the resulting part is at least one.
\end{proof}

The most significant result concerning the geometry of a square-tiled surface generating an ST-curve is \cite[Cor. 5.15]{MollerShimuraTeich}.  Without this, there would be no hope of reducing the problem to a finite checkable list.  We restate it here because it is fundamental to all of the computer calculations.

\begin{corollary}[M\"oller]
\label{MCor515}
An ST-curve generated by a surface $(X, \omega)$ can only exist in the strata listed in Table \ref{StrataListTable}.  If $\pi_{opt}: X \rightarrow E$ is the optimal torus cover, then its degree $d_{opt}$ is also given in the same table in the second column.
\end{corollary}

\begin{table}[ht]
\centering
\begin{tabular}{c|c}
Stratum & $d_{opt}$ \\
\hline
$\cH(1, 1, 1, 1)$ & $2$ \\
\hdashline
$\cH(1, 1, 1, 1, 1, 1)$ & $4$ \\
$\cH(2, 2, 2)$ & $3$ \\
\hdashline
$\cH(1, 1, 1, 1, 1, 1,2)$ & $36$ \\
$\cH(1, 1, 1, 1, 2,2)$ & $18$ \\
$\cH(1, 1, 2, 2,2)$ & $12$ \\
$\cH(2, 2, 2,2)$ & $9$ \\
$\cH(1, 1, 1, 1, 1,3)$ & $16$ \\
$\cH(1, 1, 3,3)$ & $8$ \\
$\cH(1, 1, 1, 1,4)$ & $10$ \\
\end{tabular}
\caption{List of Strata Possibly Containing ST-Curves and Values of $d_{opt}$: Dashed lines separate genus}
\label{StrataListTable}
\end{table}



In fact, further restrictions were given in \cite[Lem. 5.17]{MollerShimuraTeich}.  It gives a full description of how cylinders on the torus, which are bounded by marked points rather than zeros, lift to cylinders on the square-tiled surface generating an ST-curve.  We restate it here for the convenience of the reader.

\begin{lemma}[M\"oller]
\label{MThreeCasesLem}
Let $\pi_{opt}: X \rightarrow E$ be the optimal torus covering.  Let $B \subset E$ be the set of points over which $\pi_{opt}$ is branched.  In each of the possible cases of ST-curves listed in Table \ref{StrataListTable}, one of the following three possibilities holds.
\begin{enumerate}[label=(\roman*)]
\item The preimage under $\pi_{opt}$ of each cylinder in $E$ consists of only one cylinder in $X$, or
\item $B$ consists of one element only and each cylinder in $E$ has $k$ pre-images under $\pi_{opt}$ where $2 \leq k \leq g-1$, or
\item $g = 5$ and $B$ is contained in the set of 2-torsion points of $E$.  Moreover, each cylinder in $E$ has two preimages under $\pi_{opt}$.
\end{enumerate}
\end{lemma}

\begin{remark}
In fact, the reader may observe that Theorem \ref{AtMostTwoParts} below can be used to exclude Case (iii) because it is logically independent.  However, we will present another proof that addresses Case (ii) and (iii) at once and facilitates a linear exposition.
\end{remark}

\section{Exclusion of Cases (ii) and (iii)}
\label{CaseiiandiiiSect}

From \cite[Lem. 5.17]{MollerShimuraTeich} (Lemma \ref{MThreeCasesLem} above), it was clear that only finitely many ST-curves could satisfy Cases (ii) and (iii).  While it may be possible to exclude these with an involved computer search, a simple application of the formula for the sum of the Lyapunov exponents for an arithmetic Teichm\"uller curve from \cite[Cor. 7]{EskinKontsevichZorich2} suffices to exclude them.

\begin{theorem}[\cite{EskinKontsevichZorich2}]
\label{EKZArithLyapExpSum}
Let $\mathcal{M}_1$ be an arithmetic Teichm\"uller curve defined by a square-tiled surface $S_0$ of genus $g$ in some stratum $\mathcal{H}(m_1, \ldots, m_n)$.  The top $g$ Lyapunov exponents of the cohomology bundle $H^1$ over $\mathcal{M}_1$ along the Teichm\"uller flow satisfy:
$$\lambda_1 + \cdots + \lambda_g = 
\frac{1}{12} \sum_{i=1}^n \frac{m_i(m_i+2)}{m_i+1}$$
$$ + \frac{1}{\text{card}(\text{SL}_2(\mathbb{Z}) \cdot S_0)} \sum_{S_i \in \text{SL}_2(\mathbb{Z}) \cdot S_0} \sum_{\substack{\text{horiz. cyls. s.t.} \\ S_i = \sqcup \text{cyl}_{ij}}} \frac{h_{ij}}{w_{ij}},$$
where $h_{ij}$ and $w_{ij}$ are the height and width of cylinder $\text{cyl}_{ij}$, respectively.
\end{theorem}

\begin{lemma}
\label{KappaClass}
For all of the strata given in Table \ref{StrataListTable}, we have
$$\frac{1}{12} \sum_{i=1}^n \frac{m_i(m_i+2)}{m_i+1} = 1 - \frac{1}{d_{opt}}.$$
\end{lemma}

\begin{proof}
The values in Table \ref{StrataListTable} above are derived from the general formula for $d_{opt}$ given in \cite[Cor. 5.15]{MollerShimuraTeich}.  This lemma follows directly from that formula.
\end{proof}

Following the setup of \cite[$\S$ 10]{EskinKontsevichZorich2}, each square in the square-tiled surface is normalized to a unit square.  This implies that in Case (iii), $E$ consists of four unit squares.

\begin{lemma}
\label{SiegelVeechBound}
For any surface satisfying Case (ii) or (iii), there exists a direction such that
$$\sum_{\substack{\text{horiz. cyls. s.t.} \\ S_i = \sqcup \text{cyl}_{ij}}} \frac{h_{ij}}{w_{ij}} \geq \frac{2}{d_{opt}},$$
and in every direction
$$\sum_{\substack{\text{horiz. cyls. s.t.} \\ S_i = \sqcup \text{cyl}_{ij}}} \frac{h_{ij}}{w_{ij}} \geq \frac{1}{d_{opt}}.$$
\end{lemma}

\begin{proof}
We claim that in every direction $w_{ij} \leq d_{opt}$.  In Case (ii), the surface consists of exactly $d_{opt}$ squares, and in Case (iii), the surface consists of exactly $4d_{opt}$ squares.  By \cite[Lem 5.3]{MollerShimuraTeich} (see Lemma \ref{AtLeastTwoParts} below), every direction decomposes into at least two cylinders.  Therefore, in Case (ii), the squares must be divided among at least two cylinders and each cylinder has circumference at most $d_{opt}/2$.  In Case (iii), there are $4d_{opt}$ squares, but each cylinder on the torus lifts to two cylinders above.  It suffices to assume that $|B| > 1$ because the single branch point case is covered by Case (ii).  Hence, there is a direction on the torus that splits into four cylinders.  In this case, the circumference of each cylinder is at most $d_{opt}$.  In the case that there is a $1$-cylinder direction on $E$, we get the bound $w_{ij} \leq 2d_{opt}$.

Using the trivial bound $h_{ij} \geq 1$, implies that there exists a direction such that
$$\frac{h_{ij}}{w_{ij}} \geq \frac{1}{d_{opt}},$$
and in every direction 
$$\frac{h_{ij}}{w_{ij}} \geq \frac{1}{2d_{opt}}.$$
However, there are at least two cylinders in every direction by \cite[Lem 5.3]{MollerShimuraTeich} (Lemma \ref{AtLeastTwoParts} below).  Hence, there exists a direction such that
$$\sum_{\substack{\text{horiz. cyls. s.t.} \\ S_i = \sqcup \text{cyl}_{ij}}} \frac{h_{ij}}{w_{ij}} \geq \frac{2}{d_{opt}},$$
and in every direction
$$\sum_{\substack{\text{horiz. cyls. s.t.} \\ S_i = \sqcup \text{cyl}_{ij}}} \frac{h_{ij}}{w_{ij}} \geq \frac{2}{2d_{opt}} = \frac{1}{d_{opt}}.$$
\end{proof}

\begin{theorem}
\label{CaseiiAndiiiExclusion}
There do not exist ST-curves satisfying Case (ii) or (iii) of Lemma \ref{MThreeCasesLem}.
\end{theorem}

\begin{proof}
By Proposition \ref{EqToCDKZSpec}, a Teichm\"uller curve is an ST-curve if and only if the sum of its non-negative Lyapunov exponents equals $1$.  By \cite{EskinKontsevichZorich2} (Theorem \ref{EKZArithLyapExpSum} above), the sum of its non-negative Lyapunov exponents is exactly 
$$\lambda_1 + \cdots + \lambda_g = 
\frac{1}{12} \sum_{i=1}^n \frac{m_i(m_i+2)}{m_i+1} + \frac{1}{\text{card}(\text{SL}_2(\mathbb{Z}) \cdot S_0)} \sum_{S_i \in \text{SL}_2(\mathbb{Z}) \cdot S_0} \sum_{\substack{\text{horiz. cyls. s.t.} \\ S_i = \sqcup \text{cyl}_{ij}}} \frac{h_{ij}}{w_{ij}}.$$
Let $S'$ be the surface with a horizontal cylinder decomposition admitting the lower bound $2/d_{opt}$.  By Lemma \ref{SiegelVeechBound},
$$\frac{1}{\text{card}(\text{SL}_2(\mathbb{Z}) \cdot S_0)} \sum_{S_i \in \text{SL}_2(\mathbb{Z}) \cdot S_0} \sum_{\substack{\text{horiz. cyls. s.t.} \\ S_i = \sqcup \text{cyl}_{ij}}} \frac{h_{ij}}{w_{ij}}$$
$$ \geq \frac{1}{\text{card}(\text{SL}_2(\mathbb{Z}) \cdot S_0)} \left(\frac{2}{d_{opt}} + \sum_{\substack{S_i \in \text{SL}_2(\mathbb{Z})\cdot S_0 \\ S_i \not= S'}}  \frac{1}{d_{opt}}\right) > \frac{1}{d_{opt}}.$$
The last inequality, which is the crux of this argument follows from the fact that Teichm\"uller curves have finitely many cusps, each of which have finite width.  The average over the finite set will satisfy the strict inequality.

Returning to the formula for the sum of Lyapunov exponents and combining Lemma \ref{KappaClass} with this inequality yields
$$\lambda_1 + \cdots + \lambda_g =  1 - \frac{1}{d_{opt}} + \frac{1}{\text{card}(\text{SL}_2(\mathbb{Z}) \cdot S_0)} \sum_{S_i \in \text{SL}_2(\mathbb{Z}) \cdot S_0} \sum_{\substack{\text{horiz. cyls. s.t.} \\ S_i = \sqcup \text{cyl}_{ij}}} \frac{h_{ij}}{w_{ij}}$$
$$> 1 - \frac{1}{d_{opt}} + \frac{1}{d_{opt}} = 1.$$
Hence, the sum of the non-negative Lyapunov exponents of any such curve cannot be one, whence the curve is not an ST-curve.
\end{proof}

\section{The Boundary of an ST-Curve}
\label{BdSTCurveSect}

In this section, we apply the results of \cite{HuNortonGenVarFormsAbDiffs} to prove that an ST-curve cannot have a nodal surface in its boundary with three or more parts.  We prove this by showing that the period matrix in a neighborhood of a nodal surface with more than two parts cannot have a constant $(g-1) \times (g-1)$ minor along a Teichm\"uller geodesic.

We recall an elementary lemma  \cite[Lem 5.3]{MollerShimuraTeich} (cf. \cite[Lem. 6.1]{AulicinoCompDegKZ}) and restate its proof for the convenience of the reader.

\begin{lemma}
\label{AtLeastTwoParts}
Every nodal surface in the boundary of an ST-curve in genus at least two has at least two parts.
\end{lemma}

\begin{proof}
By contradiction, assume that there were a nodal surface with one part.  Since all parallel cylinders are homologous, this implies that the surface arises by pinching the core curve of a $1$-cylinder surface.  However, if the entire surface can be presented as a single cylinder, then consider a saddle connection on the top and bottom of the cylinder and the family of closed trajectories passing through them.  Since all trajectories parallel to these closed curves are homologous by Proposition \ref{AllCylsHom}, they have equal length, which would imply that the surface is a torus.
\end{proof}

\subsection{Setup for the Variational Formulas}

Let $(X, \omega)$ be a vertically periodic translation surface generating an ST-curve with $n$ vertical cylinders.  Consider the Teichm\"uller ray $\{g_t \cdot (X, \omega) | t \in [0, \infty)\}$, which converges to a nodal Riemann surface carrying a meromorphic differential $(X', \omega')$ as $t$ tends to infinity.  By Lemma \ref{TCIntegralRelPinchParams}, any point on the ST-curve in a sufficiently small neighborhood of the cusp $(X', \omega')$ is described by plumbing $(X',\omega')$ with parameters $s_j=s^{r_j}$ for $s\in\mathbb C, r_j\in\mathbb Z_+$. We denote the result of plumbing $(X',\omega')$ with parameters $s_j=s^{r_j}$ as $X^s$.


By Lemma \ref{STDualGraphCycle}, the dual graph of a nodal surface in the boundary of an ST-curve is always a cycle.  Therefore, every part has two nodes.  Let the parts of the surface be $X_1, \ldots, X_n$ labeled so that $X_i$ shares a node with $X_{i-1}$ and $X_{i+1}$ and subscripts are taken modulo $n$.  Then remove the nodes of $X_i$ and denote the resulting punctures by $q_i^-$ and $q_i^+$.  In this way, $q_i^-$ and $q_{i-1}^+$ correspond to punctures resulting from the removal of a single node, as do $q_i^+$ and $q_{i+1}^-$.  By convention, we arrange the labels so that each part has exactly one $+$ and one $-$.

Restricting to $X_i$, let $z_i^{\pm}$ be non-intersecting local coordinate neighborhoods of $q_i^{\pm}$.  Let the loops $\gamma_i^{+}$ and $\gamma_i^{-}$ be the boundary of neighborhoods of size $\sqrt{|s|^{r_i}}$ and $\sqrt{|s|^{r_{i-1}}}$ around the nodes $q_i^{+}$ and $q_i^-$, respectively.  In the plumbing construction, the loops $\gamma_i^+$ and $\gamma_{i+1}^-$ are glued via the function $I_i^+(z_i^+):=\frac{s^{r_i}}{z_{i+1}^-}$ or equivalently $I_{i+1}^-(z_{i+1}^-):=\frac{s^{r_{i}}}{z_{i}^+}$.  Since the subscript on the function $I$ always matches the subscript on the variable, we will drop the subscript in the formulas below.


\

\noindent \textbf{Basis of $H_1$}: We choose a symplectic basis of $H_1(X^s, \bZ)$ adapted to the vertically periodic direction above as follows.  The surface $X^s$ is the result of gluing $X_i$ with small neighborhoods removed for plumbing.  Take a union of symplectic bases of $H_1(X_i, \bZ)$ on each part $X_i$, which we denote by $\{a_{j_i},b_{j_i}\}$.  This can be completed to a symplectic basis of $H_1(X^s, \bZ)$ by adding homology cycles $\{a_1,b_1\}$ such that $a_1$ is the homology class of any core curve of a vertical cylinder (because they are all homologous), and let $b_1$ be the homology class of a curve crossing the heights of every cylinder.  Clearly, the homology basis on each $X_i$ persists to every surface in a neighborhood of a cusp.  By choosing $a_1$ to always be the core curve of a cylinder and $b_1$ to be a transverse cycle, we get a basis of first homology on $X^s$ for all $s$.

\

\noindent \textbf{Basis of $H^{1,0}$}: We fix a normalized basis of differentials relative to this basis of $H_1(X, \bZ)$. We index the basis so that $\{\theta_{j_i}\}_j$ is the collection of differentials which span $H^{1,0}(X_i)$.  Clearly each differential is supported on exactly one part of the nodal surface because $\theta_{j_i}$ for each $j_i$ evaluates to $1$ on exactly one $a$-cycle, and therefore on all other parts, it evaluates to zero on every $a$-cycle.  We complete the basis by letting $\theta_1$ be the differential (up to complex constant) giving the flat structure on $X_s$ of the Teichm\"uller curve.

\subsection{The Jump Problem for ST-Curves}

Let $\Omega$ be any differential with at most simple poles at the nodes of $X_i$ with residues $\pm r$ satisfying the matching residue condition (namely the residues at each node are equal up to sign -- pairs of pre-images of nodes come with opposite signs).  Clearly any differential in the basis of $H^{1,0}$ above satisfies this condition.  Consider the following recursively defined sequence of differentials expressed \emph{locally} in the neighborhood of each node:




\begin{equation}
\label{XiiExpan}
\begin{array}{llcl}
k=0: & \xi_i^{\pm,(0)}(z_i^\pm) & := & \Omega(z_i^\pm)-\pm\frac{rdz_i^\pm}{z_i^{\pm}};\\
k>0: &  \xi_i^{\pm,(k)}(z_i^\pm) & := & \int_{z_i^+\in\gamma_i^{+}}\bK_i(z_i^{\pm}, z_i^+)\cdot \xi_{i+1}^{-,(k-1)}(\frac{s^{r_i}}{z_i^+}) \\
& & & \qquad + \, \int_{z_i^-\in\gamma_i^{-}}\bK_i(z_i^\pm,z_i^-)\cdot \xi_{i-1}^{+,(k-1)}(\frac{s^{r_{i-1}}}{z_i^-}).
\end{array}
\end{equation}

Note that 
$$\xi_{i+1}^{-,(k-1)}\left(\frac{s^{r_i}}{z_i^+}\right)=I^*\xi_{i+1}^{-,(k-1)}(z_{i+1}^-)$$
and similarly 
$$\xi_{i-1}^{+,(k-1)}\left(\frac{s^{r_{i-1}}}{z_i^-}\right)=I^*\xi_{i-1}^{+,(k-1)}(z_{i-1}^+)$$
are the pull-backs of the local differentials on the neighboring components at each node. The local differentials $\xi_i^{\pm,(k)}(z_i^\pm)$ depend parametrically on the plumbing parameters, both in terms of the loops $\gamma_i^{\pm}$ and the pull-back $I^*$.  Let 
$$\xi_i^{\pm} = \sum_{k=0}^\infty \xi_i^{\pm,(k)},$$
which exists because the series converges by \cite[Lem. 3.2]{HuNortonGenVarFormsAbDiffs}.  On $X_i$ we have therefore defined a local differential $\xi_i^+(z_i^+)$ in the neighborhood of $q_i^+$ and similarly $\xi_i^-(z_i^-)$ at $q_i^-$.

Denote by $\hat X_i^s$ the Riemann surface with boundary given by removing two open punctured discs from $X_i$ with boundaries given by $\gamma_i^{\pm}$.  Using the local differentials $\xi_i^{\pm}$ defined above, we write a global differential on $\hat X_i^s$


\begin{equation}
\label{EtaiSemiGeneral}
\eta_i(z):=\int_{z_i^-\in\gamma_i^-}K_i(z,z_i^-)\xi_{i-1}^+\left(\frac{s^{r_{i-1}}}{z_i^-}\right) + \int_{z_i^+\in\gamma_i^+}K_i(z,z_i^+)\xi_{i+1}^-\left(\frac{s^{r_i}}{z_i^+}\right).
\end{equation}

\begin{theorem}[\cite{HuNortonGenVarFormsAbDiffs}]
\label{HNThm3_3}
When the dual graph for $X'$ is a ring, the parametrically constructed differentials $\Omega+\eta_i$ define a smooth differential on the plumbed Riemann surface $X^{s}$, i.e. they match along seams used for gluing in the plumbing construction. 

In addition the differentials $\eta_i$ satisfy the $L^2$ bound from \cite[Thm. 3.3]{HuNortonGenVarFormsAbDiffs}, which implies that in every compact subset of $X_i-\{q_i^-,q_i^+\}$ the differential $\Omega+\eta_i$ converges to $\Omega_i$.
\end{theorem}

\subsection{Variational Formulas}

The particular jump problem relevant to the proof of Theorem~\ref{AtMostTwoParts} is the following: 
We solve the jump problem for the case where $X_i$ carries any differential in the basis of $H^{1,0}(X_i)$ described above and $X_j$ carries the zero differential, for $j \not= i$.  For simplicity we drop indices and denote the differential on $X_i$ by $\theta$.  We will call this \emph{the jump problem for a ring} with respect to $\theta$ on $X_i$.

It is important to note that the plumbing was done with respect to the flat structure, namely $\theta_1$ dual to the vanishing cycle. Solving the jump problem for a ring with respect to $\theta\in H^{1,0}(X_i)$ on $X_i$ is a non-trivial problem -- namely the jumps are non-zero because clearly $\theta$ does not match with the zero differential along the seam. Thus we are actually solving a jump problem with non-trivial jump when we investigate the $(g-1)\times (g-1)$ part of the period matrix under consideration.

The following corollary is a direct application of Theorem \ref{HNThm3_3} to the situation where all components have vanishing differentials except $X_i$.

\begin{corollary}
\label{DiffFormForRing}
The solution to the jump problem for a ring with respect to $\theta$ on $X_i$ has the following form on $\hat X_{i+1}^s$.  Let $\hat r_i = r_1 + \cdots + r_{i-1} + r_{i+1} + \cdots + r_n$.  If $r_i < \hat r_i$, then
\begin{equation}\label{smoothvHatGreater}
\begin{array}{ll}
\eta_{i+1}(z) = & (-1)s^{r_i}\omega_{i+1}(z, q_{i+1}^-)\theta(q_i^+) + O(|s|^{r_i+1}).
\end{array}
\end{equation}
If $\hat r_i < r_i$, then
\begin{equation}\label{smoothv2HatLess}
\begin{array}{ll}
\eta_{i+1}(z) = & (-1)^{n-1} s^{\hat r_i}\omega_{i+1}(z, q_{i+1}^+) \prod_{\ell \not= i, i+1} \omega_{\ell}(q_{\ell}^-, q_{\ell}^+) \theta(q_i^-) + O(|s|^{\hat r_i+1}).
\end{array}
\end{equation}
If $\hat r_i = r_i$, then
\begin{equation}\label{smoothv2HatEq}
\begin{array}{ll}
\eta_{i+1}(z) = & (-1)s^{r_i}\omega_{i+1}(z, q_{i+1}^-)\theta(q_i^+) \\
& + (-1)^{n-1} s^{\hat r_i}\omega_{i+1}(z, q_{i+1}^+) \prod_{\ell \not= i, i+1} \omega_{\ell}(q_{\ell}^-, q_{\ell}^+) \theta(q_i^-) + O(|s|^{r_i+1}).
\end{array}
\end{equation}
\end{corollary}

\begin{proof}
If $r_i<\hat r_i$, we would like to argue that the lowest order term in equation (\ref{EtaiSemiGeneral}) comes from the expression $\int_{z_{i+1}^-\in\gamma_{i+1}^-}K_{i+1}(z,z_{i+1}^-)\xi_{i}^+\left(\frac{s^{r_i}}{z_{i+1}^-}\right)$. First note that the lowest order term of $\xi_i^+(z_i^+)$ is $\xi_i^{+,(0)}(z_i^+)=\theta(z_i^+)$. The integral is evaluated by applying the Cauchy integral formula, and the result is $\int_{z_{i+1}^-\in\gamma_{i+1}^-}K_{i+1}(z,z_{i+1}^-)\xi_{i}^+\left(\frac{s^{r_i}}{z_{i+1}^-}\right)=(-1)s^{r_i}\omega_{i+1}(z, q_{i+1}^-)\theta(q_i^+) + O(|s|^{r_i+1})$, see \cite[Prop. 3.4]{HuNortonGenVarFormsAbDiffs} for details. 

In order to investigate the second term in equation equation (\ref{EtaiSemiGeneral}), namely $\int_{z_{i+1}^+\in\gamma_{i+1}^+}K_{i+1}(z,z_{i+1}^+)\xi_{i+2}^-\left(\frac{s^{r_{i+1}}}{z_{i+1}^+}\right)$, notice that $\xi_{i+2}^{-,(k)}(z_{i+1}^+)$ will contain terms that are the result of pulling back $\xi_i^{\pm,(0)}(z_i^\pm)=\theta(z_i^\pm)$ along the shortest paths connecting $X_i$ to $X_{i+2}$ in either the clockwise or counterclockwise direction in the dual graph. In the clockwise direction, one would pass through the node $q_i^+\sim q_{i+1}^-$ and then through the node $q_{i+1}^+\sim q_{i+2}^-$ which therefore contains at least a power of $s^{r_i+r_{i+1}}$ which is therefore of a higher order and can be disregarded when computing the lowest order term. Additionally $\xi_{i+2}^{-,(n-2)}(z_{i+1}^+)$ contains term which result from pulling back $\xi_i^{-,(0)}(z_i^-)$ all the way around the dual graph in the counter clockwise direction. Thus $\xi_{i+2}^-\left(\frac{s^{r_{i+1}}}{z_{i+1}^+}\right)$ has a lowest order term on the order of $s^{\hat r_i}$ which by assumption is of higher order than $s^{r_i}$.

Similarly if $\hat r_i<r_i$, then we evaluate
$$\int_{z_{i+1}^+\in\gamma_{i+1}^+}K_{i+1}(z,z_{i+1}^+)\xi_{i+2}^-\left(\frac{s^{r_{i+1}}}{z_{i+1}^+}\right)$$
$$=(-1)^{n-1} s^{\hat r_i}\omega_{i+1}(z, q_{i+1}^+) \prod_{\ell \not= i, i+1} \omega_{\ell}(q_{\ell}^-, q_{\ell}^+) \theta(q_i^-) + O(|s|^{\hat r_i+1})$$
exactly as before using the Cauchy integral formula, see\cite[Prop. 3.4]{HuNortonGenVarFormsAbDiffs}. The argument showing that $\int_{z_{i+1}^-\in\gamma_{i+1}^-}K_{i+1}(z,z_{i+1}^-)\xi_{i}^+\left(\frac{s^{r_i}}{z_{i+1}^-}\right)$ produces a term of higher order is exactly as above and follows because $\xi_i^{+,(0)}\left(\frac{s^{r_i}}{z_i^+}\right)$ has a lowest power of $s^{r_i}$.

Clearly it follows that if $\hat r_i = r_i$, the lowest order term comes from the sum of lowest order terms for each.

\end{proof}

\subsection{Bound on Parts}

In this section, we present the main technical result of the paper which is the exclusion of a nodal surface in the boundary of an ST-curve with three or more parts.  We first consider the period matrix and prove some preliminary results.  We follow the setup for the homology and space of Abelian differentials established above.  Consequentially, we focus on the $(g-1) \times (g-1)$ minor of the period matrix that is constant under deformations.  We further observe that this $(g-1) \times (g-1)$ minor of the period matrix naturally decomposes into blocks $\tau(i,j)$ of size $g_i \times g_j$, for $i, j \in \{1, \ldots, n\}$.  Each component is given by
$$\tau(i,j)_{k,\ell} = \int_{b_{\ell_j}} \theta_{k_i}(s),$$
where $ \theta_{k_i}(s)$ is the solution to the jump problem with respect to $\theta_{k_i}$ on $X_j$.  We also remark that $\tau(i,j)^{T} = \tau(j,i)$ by the symmetry of the period matrix.

\

\noindent \textbf{Notation Convention}: We recall that $1$-forms are not single-valued except at the zeros. As a result, evaluating a $1$-form at a point is nonsensical.  However, if we fix a local coordinate $z$ in a domain $U$, then we can express a differential in that local coordinate as $f(z)\,dz$. In this context it makes sense to evaluate the \emph{function} $f(z)$ at any point in $U$. In the setting of this paper, we have chosen local coordinates in the neighborhood of every node used for plumbing. Therefore one may evaluate a differential at a node in the sense explained above.  This is done throughout the literature and we will do this as well without further comment. Our goal is to show the non-vanishing of an expression involving differentials evaluated at points. Clearly this is independent of the choice of local coordinates.

\

\begin{lemma}
\label{PerMatCompCalc}
Given the setup above, if $r_i < \hat r_i$, then
\begin{equation}\label{tauHatGreater}
\begin{array}{ll}
\tau(i,i+1)_{k,\ell} = -s^{r_i} \theta_{\ell_{i+1}}(q_{i+1}^-)\theta_{k_i}(q_i^+) +  O(|s|^{r_i+1}).
\end{array}
\end{equation}
If $\hat r_i < r_i$, then
\begin{equation}\label{tauHatLess}
\begin{array}{ll}
\tau(i,i+1)_{k,\ell} = & (-1)^{n-1} s^{\hat r_i}\theta_{\ell_{i+1}}(q_{i+1}^+) \prod_{\ell \not= i, i+1} \omega_{\ell}(q_{\ell}^-, q_{\ell}^+) \theta(q_i^-) + O(|s|^{\hat r_i+1}).
\end{array}
\end{equation}
If $\hat r_i = r_i$, then
\begin{equation}\label{tauHatEq}
\begin{array}{ll}
\tau(i,i+1)_{k,\ell} = & (-1)s^{r_i}\theta_{\ell_{i+1}}(q_{i+1}^-)\theta(q_i^+) \\
& + (-1)^{n-1} s^{\hat r_i}\theta_{\ell_{i+1}}(q_{i+1}^+) \prod_{\ell \not= i, i+1} \omega_{\ell}(q_{\ell}^-, q_{\ell}^+) \theta(q_i^-) + O(|s|^{r_i+1}).
\end{array}
\end{equation}
\end{lemma}

\begin{proof}
The result follows from
$$\int_{b_{\ell}} \omega( \cdot, q) = \theta_{\ell}(q).$$
\end{proof}

\begin{theorem}
\label{AtMostTwoParts}
Every nodal surface in the boundary of an ST-curve has at most two parts.
\end{theorem}

\begin{proof}
Recall that the dual graph is a cycle (or ring) by Lemma \ref{STDualGraphCycle}. 
Since the dual graph is a cycle, the number of parts equals the number of cylinders in a periodic direction.  Recall that by Lemma \ref{DerPerMat}, for an ST-curve, there exists a $(g-1) \times (g-1)$ minor of the period matrix with zero derivative along the Teichm\"uller ray defined above.  


We follow the conventions above. By Lemma \ref{PerMatCompCalc}, in each block $\tau(i, i+1)$ there exists $\ell_0$ and $k_0$ such that $\theta_{\ell_0}(q_{i+1}^-) \not= 0$ and $\theta_{k_0}(q_{i}^+) \not= 0$.  Since the two terms must cancel because the coefficient of every non-constant term is zero by assumption, we have $\hat r_i \leq r_i$ for all $i$.  This implies $\sum_j r_j \leq 2r_i$.  Since this is true for each $i$, we can sum each inequality over all $n$ values of $i$ to get
$$n\sum r_i \leq 2 \sum r_i.$$
This implies $n \leq 2$ and Lemma \ref{AtLeastTwoParts} implies $n \geq 2$.
\end{proof}

\begin{remark}
One would hope to exploit higher order terms of the expansion to draw more conclusions about the geometry of the cusps of ST-curves in the only remaining case, namely $n=2$ and $r_1=r_2$.  Unfortunately, this proves inconclusive.  We simplified the expression for the $s^{2r_1}$ coefficient in the period matrix expansion near a nodal curve whose dual graph is a ring with two components and have found some geometric characterizations for the locations of the nodes. The translation of this information into flat geometry is unclear as well as whether this information could help handle the possible remaining cases for an ST-curve in genus $g\geq 5$. 
This necessitates the reliance on a computer in Section~\ref{TwoTorPtSect} to prove Theorem~\ref{MainThm}.
\end{remark}

\subsection{Cylinder Decompositions}

Recall that after normalizing the torus to a square, the rational slopes classify the periodic directions.  Furthermore, we continue to follow the convention that a torus with marked points splits into cylinders in the periodic directions where the boundaries of the cylinders are determined by marked points.

\begin{lemma}
\label{Torus2Torsion}
Let $E$ be a torus with $s > 2$ marked points.  If $E$ splits into exactly two cylinders in every direction, then the marked points are contained in the set of $2$-torsion points.  In particular, $s \in \{3,4\}$.
\end{lemma}

\begin{proof}
In this proof, we normalize the torus to a unit square and regard it as the unit square in $\bR^2$ for the purpose of discussing horizontal and vertical lines.  Without loss of generality, let one of the points $p_1$ lie at the origin and a second point $p_2$ lie on the horizontal line through the origin.  If a third point lies on the horizontal line, then the vertical direction splits into at least three cylinders.  Therefore, all remaining marked points lie on vertical lines passing through $p_1$ and $p_2$.  The two cylinder assumption implies that there exists a third marked point $p_3$.  Otherwise, the horizontal direction would consist of a single cylinder.  Let $p_3$ be a marked point lying above $p_1$ or $p_2$.  Then all marked points must be contained in the horizontal lines through $p_1$ and $p_3$.  Otherwise, there would be more than two cylinders in the horizontal direction.  However, the intersection of two horizontal and two vertical lines is exactly a set of four points $B'$.  Hence, the set of $s$ marked points must be a subset of $B'$ consisting of at least three points.  Since the two unipotent matrices with a $1$ on the off-diagonal preserve $E$ and fix points lying on a horizontal or vertical line, the set $B'$ must be a set of four points that is invariant under the action by these matrices.  Only the $2$-torsion points of the torus have this property.
\end{proof}

\begin{proposition}
\label{STCurve2Torsion}
Let $M \in \cH(\kappa)$ be a translation surface generating an ST-curve.  Consider the optimal covering $\pi_{opt}: M \rightarrow E$ from Corollary \ref{MCor515}.  Then $\pi_{opt}$ is a covering of the torus $E$ branched over three or four of the $2$-torsion points.
\end{proposition}

\begin{proof}
By Lemma \ref{MThreeCasesLem}, $M$ satisfies one of three cases.  By Theorem \ref{CaseiiAndiiiExclusion}, $M$ necessarily satisfies Case (i).  Therefore, every cylinder on the torus lifts to a unique cylinder above.  Since there are no $1$-cylinder directions by Lemma \ref{AtLeastTwoParts}, branching of $\pi_{opt}$ must occur over at least three of the $2$-torsion points.  On the other hand, there are at most two cylinders in every direction because the degenerate surfaces in the boundary of the Teichm\"uller curve of $M$ have exactly two parts by Theorem \ref{AtMostTwoParts}.  (Recall that an $n$-cylinder direction gives rise to a degenerate surface with $n$ parts under the Teichm\"uller flow.)  Again, since every cylinder on the torus lifts to a unique cylinder above, the torus $E$ must split into two cylinders in every direction in which case Lemma \ref{Torus2Torsion} implies that all branching occurs over the set of $2$-torsion points.
\end{proof}

\section{Branching Over $2$-Torsion Points}
\label{TwoTorPtSect}

The algorithm presented below will perform a brute force search to establish the following theorem.  Theorem \ref{MainThm} will follow from Theorem \ref{NonHomCylsThm} and Proposition \ref{STCurve2Torsion}.  

\begin{theorem}
\label{NonHomCylsThm}
Let $\cH(\kappa)$ be a stratum from Table \ref{StrataListTable}.  Given a degree $d_{opt}$ covering $M \in \cH(\kappa)$ of a torus branched over three or four of the $2$-torsion points, where $M$ subject to the constraints of Lemma \ref{MLem5_18} if $d_{opt}$ is even, then $M$ decomposes into exactly two homologous cylinders in every directions only if $M$ is the Eierlegende Wollmilchsau or the Ornithorynque.
\end{theorem}

In light of Proposition \ref{STCurve2Torsion}, it is clear that we have reduced the problem to a finite number of possibilities.  Indeed, a torus with $2$-torsion points marked can be realized by four squares.  Since the degree $d_{opt}$ over this torus is explicitly given in Table \ref{StrataListTable}, the reader can observe that an ST-curve is necessarily generated by a square-tiled surface consisting of no more than $36(4) = 144$ squares, which occurs in the stratum $\cH(2,1^6)$.  However, as mentioned in the introduction, any kind of complete understanding or encyclopedia of square-tiled surfaces with $144$ squares is well beyond the scope of modern computation.  Nevertheless, the geometry of a square-tiled surface generating an ST-curve is so restricted that this problem can be solved.

Typically, square-tiled surfaces are given by a pair of permutations.  Squares are labeled by $\{1, \ldots, n\}$.  A permutation in $S_n$ describes the gluings of all squares horizontally and another permutation in $S_n$ describes the gluings of all squares vertically.  This is a completely natural way to describe the resulting square-tiled surface.  However, it is computationally unworkable for the problem here.

Instead we construct another natural system of coordinates for square-tiled surfaces that permits the extraction of meaningful information in the absence of the full coordinates that would allow one to determine the specific surface.  The upshot is that by excluding partial information, we can eliminate the existence of large families of surfaces that all share that common coordinate with a single computation.  Essentially, we record the identifications of all saddle connections, partitions of cylinder circumferences in order to provide the lengths of those saddle connections, and twists of the cylinders.  Below we will describe our coordinate system, enumerate all of the cases to address, and present a pseudo-algorithm that works as a sequence of filters of increasing computational complexity that either terminates in the empty list or produces one of the known examples.  This will complete the proof of Theorem \ref{NonHomCylsThm}.

First we recall two more results of \cite{MollerShimuraTeich} and state a simple corollary that excludes one of the strata in the table.  The following lemma is not stated as a lemma in \cite{MollerShimuraTeich}, but its complete proof can be found before the statement of \cite[Lem. 5.18]{MollerShimuraTeich}.  We follow the terminology of \cite{MollerShimuraTeich}.  Let $\phi_{opt}: \pi_1(E, P) \rightarrow S_{d_{opt}}$ be the monodromy representation with respect to the covering map $\pi_{opt}$, and let $\varepsilon: S_n \rightarrow S_n/A_n$ be the usual sign of a permutation.  A branch point is \emph{odd} (resp. \emph{even}) if the monodromy of the branch point is not in the kernel of $\varepsilon$ (resp. is in the kernel of $\varepsilon$).

\begin{lemma}[M\"oller]
\label{OddBPExistsdoptEven}
If $d_{opt}$ is even and if Case (i) holds, then there exists an odd branch point.
\end{lemma}

\begin{proof}
If $d_{opt}$ is even, then the Case (i) assumption implies that the monodromy of the core curves of the cylinders lift to cycles in $S_{d_{opt}}$ that do not lie in the kernel of $\varepsilon$.  By contradiction, if every branch point were even, then the monodromy of every branch point would lie in the kernel of $\varepsilon$.  Therefore, cylinders with every slope that is given by a product of an even number of cycles must also be even, and this contradicts Lemma \ref{MThreeCasesLem}.
\end{proof}

Next we recall \cite[Lem. 5.18]{MollerShimuraTeich} without proof for the convenience of the reader.  The proof is effectively identical to that of Proposition \ref{STCurve2Torsion}.

\begin{lemma}[M\"oller]
\label{MLem5_18}
If $d_{opt}$ is even and if Case (i) of Lemma 5.17 holds, then it is possible to normalize $E$ by translating the origin such that the set of odd branch points is precisely the set of $2$-torsion points of $E$.
\end{lemma}

Already this implies the following result.

\begin{proposition}
\label{H_2_2_2_1_1Excluded}
There are no ST-curves in $\cH(2,2,2,1,1)$.
\end{proposition}

\begin{proof}
By Theorem \ref{CaseiiAndiiiExclusion}, such a curve must satisfy Case (i).  From Table \ref{StrataListTable}, $d_{opt} = 12$ is even.  By Lemmas \ref{OddBPExistsdoptEven} and \ref{MLem5_18}, every $2$-torsion point has an odd order zero above it.  This is not possible in a stratum with exactly two odd order zeros.
\end{proof}

\subsection{Coordinates for Surfaces}

By Proposition \ref{STCurve2Torsion}, every ST-curve is generated by a surface branched over the $2$-torsion points of a torus with degree $d_{opt}$ as given in Table \ref{StrataListTable}.  Furthermore, each cylinder on the torus with period $w$ lifts to a unique cylinder of circumference $w d_{opt}$ on the surface above.  We fix the convention that we consider the torus as a $2 \times 2$ grid in the plane so that the $2$-torsion points are identified by the labels
$$\{(0,0), (1,0), (0,1), (1,1)\}.$$
In particular, the horizontal cylinders have circumference two and lift to cylinders with circumference $2d_{opt}$.  We label the cylinders $C_1$ and $C_2$.

\subsubsection{$1$-Cylinder Diagrams}

We recall that since $C_1$ and $C_2$ are homologous, cutting the core curves of each separates the surface into two parts.  One could identify the boundaries in each of these parts to get two disjoint surfaces each consisting of a single cylinder.

\begin{definition}
A \emph{$1$-cylinder diagram} $\fc = [\pi_1, \pi_2] \in S_n \times S_n$ is a pair of $n$-cycles specifying how to label $n$ saddle connections on the bottom of a cylinder from \emph{left to right} by $\pi_1$ and the top of the cylinder by $\pi_2$ from \emph{right to left}.\footnote{The difference in the order of the saddle connections is a convention from Sage.  We copy it here so that the code using Sage agrees with what is written in this paper.}
\end{definition}

Observe that a $1$-cylinder diagram contains no geometric information about the cylinder.

\begin{example}
Since $1$-cylinder diagrams do not contain any geometric data and cylinders can be freely rotated and sheared, it suffices to label the saddle connections on the bottom of a cylinder by numbers in the set $\{0, \ldots, n\}$ and the top is a permutation of these numbers.  It is crucial to note that in Sage the top permutation is written in \emph{reverse order}!

For example, the unique $1$-cylinder diagram in $\cH(2)$ using Python notation will be written
$$[[[0, 1, 2], [0, 1, 2]]]$$
because the three saddle connections on the bottom of the cylinder are labeled with $[0,1,2]$ from left to right and the three saddle connections on the top of the cylinder are labeled with $[0,1,2]$ from \emph{right to left}.\footnote{The apparent extra brackets in this notation permit multiple cylinders, so we do not break this convention even though it is not relevant to this paper.}
\end{example}

Therefore, if $M$ generates an ST-curve in the stratum $\cH(\kappa)$, then $C_1$ and $C_2$ induce a partition of $\kappa = \kappa_1 \sqcup \kappa_2$ such that $\kappa_1$ is the set of zeros between the top of $C_1$ and the bottom of $C_2$ and $\kappa_2$ is the set of zeros between the top of $C_2$ and the bottom of $C_1$.  Since ST-curves lie in genus at most five, the $1$-cylinder diagrams that we consider lie in genus at most four and these can be readily enumerated by any computer, or even by hand in most cases.

\subsubsection{Partitions}

Since cylinder diagrams do not come with any length information for the individual saddle connections by definition, we must determine all possible lengths of the saddle connections.  If a $1$-cylinder diagram has $n$ saddle connections and the cylinder has circumference $2d_{opt}$, then the lengths are given by considering all ordered partitions of $2d_{opt}$ into $n$ positive integers.  Each ordered partition will be written as a tuple $\cP = (p_1, \ldots, p_n)$.  In fact, this can be dramatically improved by setting conventions to eliminate equivalent or inadmissible surfaces.

\subsubsection{Conventions}

Since each point in moduli space is an equivalence class under the action of the mapping class group, we are free to cut and rearrange the surface as we wish.  Furthermore, Teichm\"uller curves are $\splin$-invariant, so we are free to rearrange the $2$-torsion points on the torus by acting on the surface by elements in $\text{SL}_2(\bZ)$, which includes shearing the cylinders by any desired amount.

Let $\{\tau_0, \ldots, \tau_{n_1 - 1}\}$ denote the set of saddle connections in the bottom of $C_1$ and let $\{\sigma_0, \ldots, \sigma_{n_2 - 1}\}$ denote the set of saddle connections in the bottom of $C_2$.  Let $\fc_1$ denote the $1$-cylinder diagram describing the identifications between the bottom of $C_1$ and the top of $C_2$.  Let $\fc_2$ denote the $1$-cylinder diagram describing the identifications between the bottom of $C_2$ and the top of $C_1$.  Let $\cP_1 = (t_0, \ldots, t_{n_1 - 1})$ denote an ordered partition of $2d_{opt}$ into $n_1$ numbers, so that the bottom of $C_1$ (and top of $C_2$) contain exactly $n_1$ saddle connections.  Likewise, let $\cP_2 = (s_0, \ldots, s_{n_2 - 1})$ denote an ordered partition of $2d_{opt}$ into $n_2$ numbers, so that the bottom of $C_2$ (and top of $C_1$) contain exactly $n_2$ saddle connections.

Cut and reglue $C_1$ and shear it if necessary so that when depicted as a rectangle the top left and bottom left of the first square in $C_1$ contain a zero, and in both cases is incident on the left with a longest saddle connection in each, which we declare to be $\sigma_0$ and $\tau_0$ with lengths $s_0$ and $t_0$, respectively.  Finally, rearrange $C_2$ if necessary so that the bottom left of the first square in $C_2$ contains a zero and the saddle connection to its right is exactly $\sigma_0$.  At this point, all choices have been exhausted and we cannot choose any conventions for the top of $C_2$.  Nevertheless, the location of $\tau_0$ on the top of $C_2$ is dramatically restricted as we will see below.

\subsubsection{Surface Coordinates}

The reader is advised to consult Figure \ref{SurfCoordSchematicFig} throughout this section.  With the cylinders arranged as above, we label the individual squares.  Label the squares of $C_1$ by $\{1, \ldots, 2d_{opt}\}$ and the squares of $C_2$ by $\{2d_{opt} + 1, \ldots, 4d_{opt}\}$.  Given two $1$-cylinder diagrams and partitions all ordered to follow the conventions above, the only piece of data needed to determine the surface is the starting point of a saddle connection on the top of $C_2$.  We choose the left endpoint of the saddle connection $\tau_0$ as the saddle connection that we wish to record.  For reasons that will be clarified below, define the value $t_{start}$ to be the non-negative integer such that the square $2d_{opt} + 2s_0 + t_{start}$ has the left endpoint of $\tau_0$ in its upper left corner.  It will be be proven below that $t_{start}$ can only realize a very small range of values.  With this, the surface coordinates we consider will be
$$((\cP_1, \fc_1), t_{start}, (\cP_2, \fc_2)).$$

For example, the following coordinates in Python notation produces a surface in $\cH(2,2,2)$.
$$(((1, 1, 1, 1, 1, 1), [[0, 5, 3, 1, 2, 4], [0, 5, 3, 1, 2, 4]]), 0, ((2, 2, 2), [[0, 1, 2], [0, 1, 2]]))$$
In fact, this is the Ornithorynque.

We do \emph{not} claim that two different coordinates give rise to two different surfaces.  In fact, we will see that there are numerous redundancies.  Nevertheless, sufficiently many have been eliminated with these conventions to make the problem computationally feasible.

\begin{figure}[htb]
  \centering
  \begin{tikzpicture}[scale=1]
    \draw (0,0) -- (10,0) -- (10,1) -- (0,1) -- (0,0);
    \draw (0,3) -- (10,3) -- (10,4) -- (0,4) -- (0,3);
    \draw [dashed] (1,0) -- (1,1);
    \draw [dashed] (2,0) -- (2,1);
    \draw [dashed] (9,0) -- (9,1);
    \draw [dashed] (1,3) -- (1,4);
    \draw [dashed] (2,3) -- (2,4);
    \draw [dashed] (9,3) -- (9,4);
    \foreach \x in {(0,0), (3,0), (0,1), (2,1), (10,0), (10,1)} {\node[circle, fill=black] at \x {};}
    \foreach \x in {(0,3), (2,3), (10,3), (5,4), (8,4)} {\node[circle, fill=black] at \x {};}
    \foreach \x in {(0,4), (4,4)} {\node[rectangle, fill=black] at \x {};}
    \node[circle] at (1/2,1/2) {$1$};
    \node[circle] at (1+1/2,1/2) {$2$};
    \node[circle] at (9+1/2,1/2) {$2d_{opt}$};
    \node[circle] at (1/2,3+1/2) {\tiny $2d_{opt} + 1$};
    \node[circle] at (1+1/2,3+1/2) {\tiny $2d_{opt} + 2$};
    \node[circle] at (9+1/2,3+1/2) {$4d_{opt}$};
    \node[circle, label=above:$\tau_0$] at (13/2,4) {};
    \node[circle, label=below:$\tau_0$] at (3/2,0) {};
    \node[circle, label=above:$\sigma_0$] at (1,1) {};
    \node[circle, label=below:$\sigma_0$] at (1,3) {};
    \node[circle, label=left:$C_1$] at (0,1/2) {};
    \node[circle, label=left:$C_2$] at (0,3+1/2) {};
    \draw [decoration={brace,mirror,amplitude=7}, decorate] (0,-.5) --node[below=3mm]{$t_0$} (3,-.5);
    \draw [decoration={brace,mirror,amplitude=7}, decorate] (0,2.5) --node[below=3mm]{$s_0$} (2,2.5);
    \draw [decoration={brace,amplitude=7}, decorate] (0,4.5) --node[above=3mm]{$2s_0$} (4,4.5);
    \draw [decoration={brace,amplitude=7}, decorate] (4,4.5) --node[above=3mm]{$t_{start}$} (5,4.5);
  \end{tikzpicture}
\caption{Schematic of the Surface Coordinates Used in the Algorithm: Circle vertices represent zeros and square vertices could represent zeros or regular points}
\label{SurfCoordSchematicFig}
\end{figure}
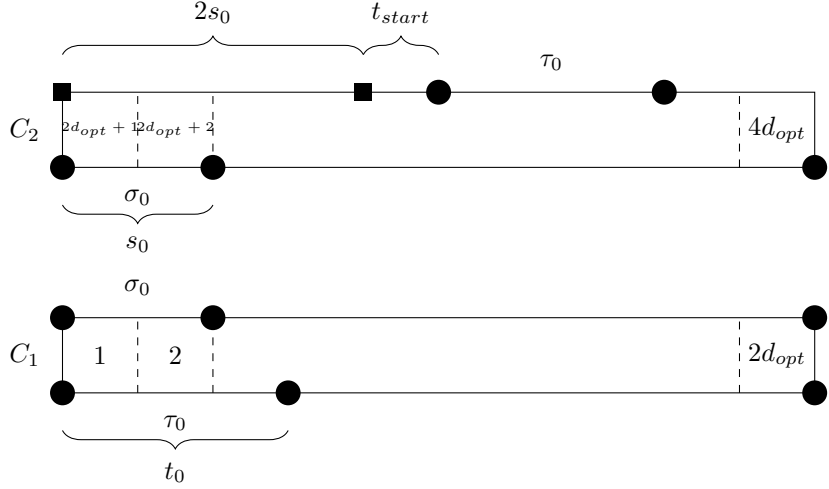

\subsection{Enumeration of the Cases}

Since all of the branching occurs over the $2$-torsion points, there are very few possibilities for which zeros can lie over which $2$-torsion points.  Nevertheless, in most cases even after fixing conventions the possibilities are not unique in each stratum.  We enumerate all of them here and fix conventions.

\subsubsection{$d_{opt}$ Even}

In this case, \cite[Lem. 5.18]{MollerShimuraTeich} (Lemma \ref{MLem5_18} above) implies that each $2$-torsion point must have an odd order zero above it.  The remaining zeros must be partitioned into even numbers and distributed over the $2$-torsion points.  Since the sum of the orders of the zeros is at most eight, which is realized in genus five, and at least total order four must be absorbed by the odd branching over the $2$-torsion points, there are either one or two of the $2$-torsion points that have additional branching over them.  Furthermore, observe from Table \ref{StrataListTable} that every stratum with $d_{opt}$ even has at least two simple zeros.  We make the following crucial observation, which follows from the fact that we are free to rotate and rearrange the surface.

\

\noindent \textbf{Observation}: The surface can be rotated and rearranged so that exactly one $2$-torsion point in each set $\{(0,0), (1,0)\}$ and $\{(0,1), (1,1)\}$ has a simple zero above it and no other branching.

\

Combine with \cite[Lem. 5.6]{MollerShimuraTeich}, which was stated for ST-curves even though the proof establishes the following general lemma without modification.  This will dramatically reduce the number of coordinates for surfaces that we need to consider.

\begin{lemma}
\label{SimpZeroToSelfExc}
A cylinder for a $1$-cylinder diagram in a stratum without marked points cannot contain a saddle connection in its boundary from a simple zero to itself.
\end{lemma}

\begin{proof}
By contradiction, let $\sigma$ be a saddle connection from a simple zero to itself on the top and bottom of a $1$-cylinder surface.  Observe that there is no admissible identification of the saddle connections incident with $\sigma$ that would result in the endpoints of $\sigma$ corresponding to a simple zero.
\end{proof}

\noindent \underline{$\cH(1^4)$}: There is exactly one simple zero over each $2$-torsion point.

\

\noindent \underline{$\cH(1^6)$}: There is one simple zero over each $2$-torsion point.  The remaining two simple zeros must occur over the same point, which we place at the origin without loss of generality.

\

\noindent \underline{$\cH(2, 1^6)$}: There is one simple zero over each $2$-torsion point, and either
\begin{itemize}
\item all three remaining zeros descend to a point in $\{(0,0), (1,0)\}$,\footnote{We emphasize here that this one case is the most computationally intensive of all of the cases by several orders of magnitude.  If run on a single processor, we estimate that it could take 7 years to exclude it.  However, the search is fully parallelizable.} or
\item the two remaining simple zeros occur over one of the points in $\{(0,0), (1,0)\}$ and the double zero occurs over one of the points in $\{(0,1), (1,1)\}$.
\end{itemize}

\

\noindent \underline{$\cH(2, 2, 1^4)$}: There is one simple zero over each $2$-torsion point, and either
\begin{itemize}
\item both double zeros descend to a point in $\{(0,0), (1,0)\}$, or
\item one double zero descends to one of the points in $\{(0,0), (1,0)\}$ and the other double zero descends to one of the points in $\{(0,1), (1,1)\}$.
\end{itemize}

\

\noindent \underline{$\cH(3, 1^5)$}: There is a simple zero over at least three of the $2$-torsion points.  The fourth $2$-torsion point
\begin{itemize}
\item has a triple zero above it and no other zero, and it descends to a point in $\{(0,1), (1,1)\}$.  The two remaining simple zeros descend to one of the points in $\{(0,0), (1,0)\}$.
\item The fourth $2$-torsion point has a triple zero above it as well as both of the remaining simple zeros, and they descend to one of the points in $\{(0,0), (1,0)\}$.
\end{itemize}

\

\noindent \underline{$\cH(1, 1, 3,3)$}: The branching is unique in this case because each $2$-torsion point must have an odd order zero above it.  However, we choose to arrange it so that one of the triple zeros lies over exactly one of the points in $\{(0,0), (1,0)\}$ and the other over one of the points in $\{(0,1), (1,1)\}$.

\

\noindent \underline{$\cH(4, 1^4)$}: There is one simple zero over each $2$-torsion point.  The quadruple zero can be arranged so that it descends to one of the points in $\{(0,0), (1,0)\}$.

\subsubsection{$d_{opt}$ Odd}

\noindent \underline{$\cH(2^3)$}: Three of the $2$-torsion points must have a double zero above them by Proposition \ref{STCurve2Torsion}.  Therefore, we can arrange the surface so that each of $(0,0)$ and $(1,0)$ have a double zero above them.

\

\noindent \underline{$\cH(2^4)$}: Three of the $2$-torsion points must have a double zero above them by Proposition \ref{STCurve2Torsion}.  Therefore, we can arrange the surface so that either 
\begin{itemize}
\item each $2$-torsion point has exactly one double zero above it, or
\item three double zeros descend to $\{(0,0), (1,0)\}$.
\end{itemize}

\subsection{The Algorithm for $d_{opt}$ Even}

\begin{lemma}
\label{FourOddSC}
Given the conventions and choices above in the case when $d_{opt}$ is even, there are exactly four saddle connections between the boundaries of each cylinder with odd length and they correspond to all of the saddle connections incident with a simple zero.  All the rest of the saddle connections have even length.
\end{lemma}

\begin{proof}
The total angle around a simple zero is $4\pi$.  Therefore a simple zero must have two copies on each boundary of a cylinder.  Hence, it is incident with at most four saddle connections, and by Lemma \ref{SimpZeroToSelfExc}, it must be incident with exactly four saddle connections.  Since the simple zero is the unique zero over one of the $2$-torsion points by the arrangement of the cases above, each of these saddle connections must have odd length because they join a zero over one $2$-torsion point to a zero over a different $2$-torsion point.  On the other hand, all other zeros in the boundary of a cylinder lie over the same $2$-torsion point and the base torus has width two.  Therefore, all other saddle connections have even length.
\end{proof}

The following corollary, when combined with Proposition \ref{STCurve2Torsion}, provides a simplified proof of \cite[Cor. 5.19]{MollerShimuraTeich}.

\begin{corollary}
\label{ExcludeH1t6}
There are no ST-curves in $\cH(1^6)$.
\end{corollary}

\begin{proof}
Recall that $d_{opt} = 4$, so the circumference of a cylinder is $8$ in this case.  The bottom of $C_1$ has four simple zeros in its boundary and $1$-cylinder diagrams in $\cH(1^4)$ have eight saddle connections in their boundary.  There does not exist a partition of eight into eight positive integers such that exactly four, and no more, are odd.  This contradicts Lemma \ref{FourOddSC}.
\end{proof}

\begin{remark}
If the boundary between two cylinders consists of exactly two simple zeros, then we need only consider partitions of $2d_{opt}$ into four odd numbers.
\end{remark}

\noindent \textbf{Step 1}: For each $1$-cylinder diagram, label the vertices and construct lists of the vertex labels.  For each simple zero construct a binary list that will correspond to saddle connection lengths modulo two by associating a $1$ to each saddle connection incident with a fixed simple zero.  The functions for doing this are contained in the notebook \verb|cyl_diag_fcns.ipynb|.  The function that produces the list of all possible binary lists for each of the $1$-cylinder diagrams is \verb|strat_odd_sc|.

\

We state a trivial lemma that is actually very powerful because it is so computationally efficient.  In fact, it is sufficiently strong that given the setup so far it will completely determine both the genus three and genus four examples.

\begin{lemma}[The Window Lemma]
\label{WindowLem}
There does not exist a closed regular trajectory from a saddle connection to itself that crosses the core curves of $C_1$ and $C_2$ exactly once.
\end{lemma}

\begin{proof}
Shear the surface so that the trajectory becomes vertical.  Then the closed trajectory has length two, but vertical cylinders must have circumference $2d_{opt} > 2$.
\end{proof}

The following corollary strongly restricts the range of $t_0$ and $s_0$ in the partitions above.

\begin{corollary}
\label{s0t0Inequality}
Let $s_0$ and $t_0$ be elements in the partition as above.  Then $t_{start}$ is even and satisfies 
$$2d_{opt} - 2t_0 - 2s_0 \geq t_{start} \geq 0.$$
\end{corollary}

\begin{proof}
See Figure \ref{WindowLemmaCorFig} for a picture proof of this corollary.  Recall $\sigma_0$ (resp. $\tau_0$) is the saddle connection with length $s_0$ (resp. $t_0$).  Consider square $1$, which has $\tau_0$ on its bottom and $\sigma_0$ on its top.  Observe that there are trajectories from the bottom of square $1$ passing through $\sigma_0$ to the tops of all squares labeled $\{2d_{opt}, \ldots, 2d_{opt} + 2s_0\}$.  On the other hand, if we consider square $t_0$, which has the right endpoint of $\tau_0$ on its lower right, then there are trajectories from that square passing through $\sigma_0$ to all squares labeled $\{4d_{opt} - t_0 + 1, \ldots, 4d_{opt}\}$.  Therefore, the starting point of $\tau_0$ must lie somewhere after or including square $2d_{opt} + 2s_0$ and end on square at most $4d_{opt} - t_0$ by the Window Lemma.  Hence, $2d_{opt} - 2t_0 - 2s_0 \geq t_{start} \geq 0$ holds.

The fact that $t_{start}$ is even follows from the fact that the lower left corner of square $1$ descends to the origin and $t_{start}$ provides the location of the same zero.
\end{proof}

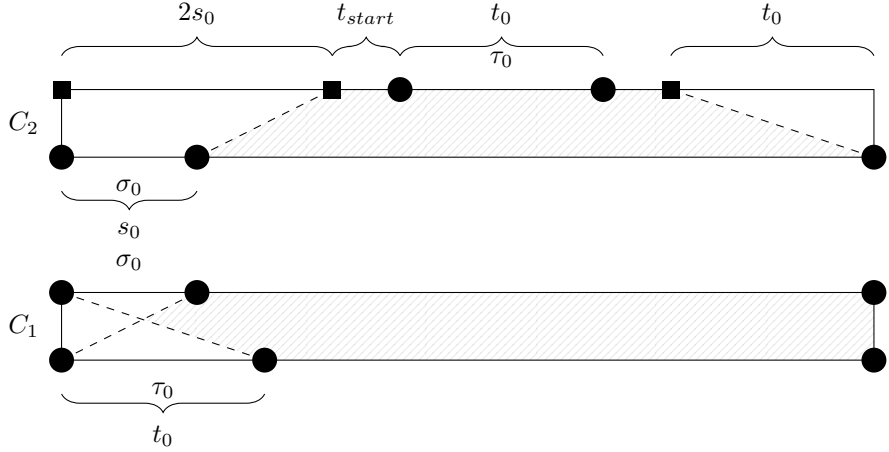
\begin{figure}[htb]
  \centering
  \begin{tikzpicture}[scale=.9]
    \draw (0,0) -- (12,0) -- (12,1) -- (0,1) -- (0,0);
    \draw (0,3) -- (12,3) -- (12,4) -- (0,4) -- (0,3);
    \draw [dashed] (0,0) -- (2,1);
    \draw [dashed] (2,3) -- (4,4);
    \draw [dashed] (3,0) -- (0,1);
    \draw [dashed] (12,3) -- (9,4);
    \foreach \x in {(0,0), (3,0), (0,1), (2,1), (12,0), (12,1)} {\node[circle, fill=black] at \x {};}
    \foreach \x in {(0,3), (2,3), (12,3), (5,4), (8,4)} {\node[circle, fill=black] at \x {};}
    \foreach \x in {(0,4), (4,4), (9,4)} {\node[rectangle, fill=black] at \x {};}
    \node[circle, label=above:$\tau_0$] at (13/2,4) {};
    \node[circle, label=below:$\tau_0$] at (3/2,0) {};
    \node[circle, label=above:$\sigma_0$] at (1,1) {};
    \node[circle, label=below:$\sigma_0$] at (1,3) {};
    \node[circle, label=left:$C_1$] at (0,1/2) {};
    \node[circle, label=left:$C_2$] at (0,3+1/2) {};
    \draw [decoration={brace,mirror,amplitude=7}, decorate] (0,-.5) --node[below=3mm]{$t_0$} (3,-.5);
    \draw [decoration={brace,mirror,amplitude=7}, decorate] (0,2.5) --node[below=3mm]{$s_0$} (2,2.5);
    \draw [decoration={brace,amplitude=7}, decorate] (0,4.5) --node[above=3mm]{$2s_0$} (4,4.5);
    \draw [decoration={brace,amplitude=7}, decorate] (4,4.5) --node[above=3mm]{$t_{start}$} (5,4.5);
    \draw [decoration={brace,amplitude=7}, decorate] (9,4.5) --node[above=3mm]{$t_0$} (12,4.5);
    \draw [decoration={brace,amplitude=7}, decorate] (5,4.5) --node[above=3mm]{$t_0$} (8,4.5);
    \fill [pattern=north east lines, pattern color=black, opacity=.3] (3,0) -- (6/5,3/5) -- (2,1) -- (12,1) -- (12,0);
    \fill [pattern=north east lines, pattern color=black, opacity=.3] (2,3) -- (4,4) -- (9,4) -- (12,3);
     \end{tikzpicture}
\caption{Proof of Corollary \ref{s0t0Inequality}: The shaded region on the top of $C_2$ represents the admissible locations of $\tau_0$}
\label{WindowLemmaCorFig}
\end{figure}

\begin{example}
\label{PartitionExample}
We go through an example that illustrates all of the restrictions on the partitions.  Consider the stratum $\cH(2,1^6)$ where $d_{opt} = 36$ and all of the extra branching occurs over a single zero following the convention above.  In this case, we consider $1$-cylinder diagrams in $\cH(2,1^4)$ and $\cH(1,1)$.  In the former case, the cylinder diagrams have $11$ saddle connections in their boundaries.  Therefore, we consider partitions
$$\sum_{i=0}^{10} t_i = 2(36) = 72.$$
Every partition of $72$ into $11$ numbers must have an element that is at least $7$.  However, we claim that the largest number in the partition cannot be $7$ in our context.  Since every partition under consideration must have exactly four odd numbers, observe that there are no partitions of $72$ with four sevens and every other number bounded above by six.  Therefore, $t_0 \geq 8$.

On the other hand, the unique $1$-cylinder diagram in $\cH(1,1)$ contains four saddle connections in its boundary.  Thus, we consider partitions of $72$ into four \emph{odd} integers.  Since $72/4 = 18$, the largest number in any admissible partition must be at least $19$, so $s_0 \geq 19$.

The largest value of $t_0$ that does not violate the inequality
$$2d_{opt} - 2t_0 - 2s_0  = 72 - 2t_0 - 38 \geq 0$$
is $17$.  Hence, $8 \leq t_0 \leq 17$.

For $s_0$, solve
$$2d_{opt} - 2t_0 - 2s_0  = 72 - 16 - 2s_0 \geq 0$$
to get that $s_0$ is at most $28$.  However, since $s_0$ is odd, we have $s_0 \leq 27$.  Hence, $s_0 \in \{19, 21, 23, 25, 27\}$.
\end{example}

\noindent \textbf{Step 2}: Create the partitions by finding the ranges of each of $s_0$ and $t_0$.  Then generate all partitions of $2d_{opt}$ into $n_1$ and $n_2$ numbers such that the first number varies over the admissible range of $s_0$ and $t_0$, respectively.  Finally, for each value of $s_0$ and $t_0$, those values bound the rest of the partition.  We also get to impose further restrictions.  Since every zero has an incoming and outgoing saddle connection, every odd number in the partition must come in a pair.  All of this is accomplished in the functions in the Sage notebook \verb|partition_functions.ipynb|.

\

Having extracted the necessary data about the location of the vertices in the cylinder diagrams and generated all possible partitions that could work with the established conventions, combine the partitions with the cylinder diagrams.  In this context, if we consider the saddle connections $\tau_i$, for example, then we call $\sigma_0$ the ``window'' and vice versa because trajectories starting at $\tau_i$ can ``see'' through $\sigma_0$ to other $\tau$ saddle connections.

\

\noindent \textbf{Step 3}: First, take each partition and reduce it modulo two to identify the odd numbers.  Denote the partition modulo two by $B_0$.  Then for each cylinder diagram consider each binary list corresponding to each simple zero as derived above.  Denote this by $B_1 = (b_0, \ldots, b_{n-1})$.  For each $B_0$, consider all cyclic permutations of $B_1$ and test if they are equal to $B_0$.  Each time we achieve an equality of the lists, record the pair of the partition with the cylinder diagram such that the bottom of the cylinder diagram has been cyclically permuted by an appropriate amount to achieve the equality.  This information is stored in files with the tag \verb|align_list| and the functions that do this are contained in the Sage notebook \verb|align_list_fcns.ipynb|.

\

\noindent \textbf{Step 4}: Observe that given a specific value for $s_0$ and $t_0$, Corollary \ref{s0t0Inequality} determines a full range of admissible values for $t_{start}$.  The next step takes maximum advantage of the Window Lemma.  Instead of simply checking that, say $\tau_0$, cannot see itself through window $\sigma_0$, we can fix a window length and check that none of the saddle connections in the bottom of cylinder $C_1$, in this case, can see themselves through the respective window.  The key is that this is actually extremely easy to check and it is encoded in the following two lines contained in the function \verb|visible_sc_check|.  (The variable \verb|s1| represents the quantity $s_0$.)

\begin{verbatim}
t_i_new_start
    = (t_i_orig_start[1] + t_i_sc_length_shift[i[0]])%width
if t_i_new_start < 2*s1 or t_i_new_start > width - 2*t_i_length:
    return False
\end{verbatim}

The starting point of $\tau_0$ depends of course on $s_0$ and $t_{start}$.  Shearing the surface and cutting and rearranging it so that saddle connection $\tau_i$ takes the place of $\tau_0$ we see that the shear was by $\sum_{j < i} t_j$.  Therefore, if the new starting point is less than $2s_0$ or larger than $2d_{opt} - 2t_i$, then it fails the Window Lemma and returns \verb|False|.  The output of partition $(t_0, \ldots, t_{n_1-1})$, with their cylinder diagrams, the value of $t_{start}$ and the window is stored in a file that includes the value of $t_0$ and the size of the Window in the filename: \verb|_align_list_| $t_0$  \verb|_visible_| Window, e.g. \verb|H_2_1_1_align_list_9_visible_8| corresponds to partitions for the seven saddle connections in the boundary of the $1$-cylinder diagrams in $\cH(2,1,1)$ starting with $9$ that cannot see through a window of length $8$.  This is repeated for the symmetric situation between the other boundaries of the cylinders.  All of these functions are contained in the Sage notebook \verb|align_list_visible_fcns.ipynb|.

\

\noindent \textbf{Step 5}: At this point we are ready to produce a full list of surfaces with complete coordinates as described in the previous section.  There is one extra piece of data that we get to take advantage of here.  In constructing the list in Step 4 above, each of the cylinder boundaries produces a $t_{start}$ value.  For the purpose of disambiguation we can refer to one as $t_{start}$ and the other as $s_{start}$.  However, this actually over-determines the surface.  By the conventions above, one of these values completely determines the other.  Therefore, when combining the lists we check to see that the value of $s_{start}$ implied by the value of $t_{start}$ is consistent with the value of $t_{start}$ that is implied by the value of $s_{start}$.  If not, the surface is not added to our final list.  It turns out that the compatibility of $t_{start}$ and $s_{start}$ is checked by an elementary arithmetic operation.

\begin{lemma}
\label{TstartSstartRelation}
$$2d_{opt} = 2s_0 + 2t_0 + s_{start} + t_{start}.$$
\end{lemma}

\begin{proof}
This follows from writing down the two presentations of the surface.  Then transform one into the other by rearranging the cylinders and using a unipotent matrix to shear the requisite amount.
\end{proof}

Given the file notation, the construction of the list is actually very simple.  For partitions starting with $t_0$, the value of $s_0$ serves as the window, and for partitions starting with $s_0$, the value of $t_0$ serves as the window.  Therefore, we consider pairs of files with exactly these symmetric tags, check $t_{start}$ and $s_{start}$ as described above, and store all of the admissible surfaces in a file titled \verb|admissible_list|.  All of these functions are contained in the Sage notebook \verb|combine_align_list_visible_fcns.ipynb|.

\

At this point, several cases return the empty list and therefore do not exist.  In fact, the most computationally difficult case - the stratum $\cH(2,1^6)$ where all of the extra branching occurs over a single point - returns empty lists after Step 4.  However, for other cases, more filters need to be added.  We proceed by looking for cylinders of length $2d_{opt}$.  Specifically, we construct the vertical permutations.  Even in the case where the \verb|admissible_list| file is as large as possible, it is still possible to perform all of the calculations in the rest of the algorithm with a standard personal computer.

\

\noindent \textbf{Step 6}: Start at square $1$ and continue vertically until we return.  The length of this permutation must be exactly $2d_{opt}$.  Take each surface in \verb|admissible_list| that yields a surface where this one vertical permutation has exactly this length and store it in the file \verb|admissible_list_vert_perm_check|.  From this second list, test all vertical permutations to make sure that there are exactly two $2d_{opt}$-cycles and store all of the surfaces that satisfy this in the file \\ \verb|admissible_list_all_vert_perm_check|.  All of these functions are contained in the Sage notebook \verb|vert_perm_check_fcns.ipynb|.

With few exceptions \verb|admissible_list_all_vert_perm_check| is the empty list.

\

\noindent \textbf{Step 7}: Finally, take horizontal and vertical permutations $h$ and $v$, respectively.  Compute the permutation $h^k v$ corresponding to the monodromy elements of cylinders with slope $1/k$ for $1 \leq k \leq d_{opt}$.  The result must always be a pair of $2 d_{opt}$-cycles.  This function is contained in the Sage notebook \verb|slope_test_fcn.ipynb|.  We do not bother to store the final result in a file because this suffices to rule out every case in genus five, so that only the genus three case remains.  (Genus four is in the next section.)

\

\noindent \textbf{Step 8}: We use the isomorphism test for square-tiled surfaces in the Sage \verb|surface_dynamics| package to reduce the lists so that a single element remains in genus three.  The function to find the sum of Lyapunov exponents for a square-tiled surface is a rigorous function that applies the formula of \cite[$\S$ 10]{EskinKontsevichZorich2} (Theorem \ref{EKZArithLyapExpSum} above) to compute the sum.  It returns the answer that the sum is $1$ and proves that the Eierlegende Wollmilchsau generates the unique ST-curve with $d_{opt}$ even.


\subsection{The Algorithm for $d_{opt}$ Odd}

In the odd case, we follow the same algorithm as in the even case.  However, in the absence of simple zeros, some of the restrictions on the partitions are completely different here.  Regardless of whether they are weaker or stronger, the number of cases is sufficiently small that all calculations can be performed nearly instantly on any computer.  Therefore, we explain the restriction on the partitions and use the same functions as in the even case for all other steps of the algorithm.

In the case where there is a unique zero above every point, observe that saddle connections from a zero to itself must have even length, while saddle connections between distinct zeros must have odd length.  Therefore, we construct binary lists as above to capture this information.  In particular, in the case where there is a single double zero between two cylinders, all three saddle connections must have even length.

This holds for all cases except in the stratum $\cH(2^4)$ where the branching occurs over three of the $2$-torsion points.  In this case each of the three zeros in the $1$-cylinder diagrams in $\cH(2^3)$ could be the isolated double zero.  Therefore, we consider all three possibilities.  Though this is redundant, the computations are so fast that we opt for redundancy over potential exclusion.  With this modification, we continue to Step 3 above and proceed as before to see that there are no ST-curves in $\cH(2^4)$ and that the Ornithorynque generates the unique ST-curve in $\cH(2^3)$.

\begin{subappendices}
\subsection{A Technical Lemma}

In executing the computer search in the stratum $\cH(2, 1^6)$ in the case where all of the extra branching lies over a single point, an elementary lemma can eliminate the need to check nearly half of the cases.  In this case, every gain is especially valuable given the number of configurations to check.  For other cases, the assumption of the lemma is not satisfied in general and therefore the lemma is useless.

\begin{lemma}
Given a partition aligned with a cylinder diagram, $(\cP_1, \fc_1)$, and a window length $s_0$, if the function \verb|visible_align_check| returns \verb|False| for the input $(\cP_1, \fc_1)$, the window $s_0$, and all admissible values of $t_{start}$, then for all windows $s > s_0$, the function also returns \verb|False|.
\end{lemma}

\begin{proof}
The variable \verb|s1| is the value of the window length $s_0$.  The value of $t_{start}$ can take on all values that it can for all other windows.  In the function \verb|visible_sc_check|, the test for failure is given by
$$\verb|t_i_new_start| <2\verb|s1|$$
or
$$\verb|t_i_new_start| > 2d_{opt} - 2\verb|t_i_length|.$$
Therefore, if the test fails because $\verb|t_i_new_start| < 2\verb|s1| = 2s_0$, then certainly it fails for all larger windows, i.e. $s > s_0$.  On the other hand, observe that \verb|t_i_length| is independent of the window.  Therefore, if the test fails for the second condition, then it does so regardless of the window.
\end{proof}

In the stratum $\cH(2, 1^6)$ in the case where all of the extra branching lies over a single point, the function \verb|visible_align_check| takes the place of \verb|visible_sc_check|.  In that case, there must be a window with size at least $19$ (see Example \ref{PartitionExample}).  After running the entire search, no surfaces were found that work with a window of size $19$.  Therefore, no larger window can work.
\end{subappendices}

\section{Classification of Affine Submanifolds with Completely Degenerate KZ-Spectrum}
\label{CDKZSpecSect}

One of the most difficult obstacles to overcome in the work of the first named author \cite{AulicinoCompDegKZ, AulicinoCompDegKZAIS} was the possible existence of Shimura-Teichm\"uller curves in genus five.  Complicated degeneration arguments were needed in both works to exclude the existence of higher dimensional orbit closures with completely degenerate KZ-spectrum.  In fact, one of the main results of \cite{AulicinoCompDegKZ} was a theorem excluding the existence of orbits with completely degenerate KZ-spectrum in genus five and six that was conditional on the conjecture that ST-curves do not exist in genus five, which is now proven in Theorem \ref{MainThm} above.  Following the work of \cite{EskinMirzakhaniMohammadiOrbitClosures} far stronger theorems could be proven.  This was the content of \cite{AulicinoCompDegKZAIS}, which proved that there are no higher dimensional orbit closures with completely degenerate KZ-spectrum.  From this, one concludes that there must be only finitely many ST-curves in genus five because otherwise, they would accumulate to an affine submanifold with higher dimension.

In this section, we give a guide for how to extract the necessary results from \cite{AulicinoCompDegKZAIS} to obtain a classification of all affine submanifolds with completely degenerate KZ-spectrum using the full classification of ST-curves established here.

\begin{proof}[Proof of Thm. \ref{FullCDKZClassThm}]
Every $\splin$-orbit closure is an immersed affine submanifold $\cM$ by \cite{EskinMirzakhaniMohammadiOrbitClosures}.  If $\cM$ has completely degenerate KZ-spectrum, then it has a rank $g-1$ Forni subspace, i.e. the monodromy of the KZ-cocycle admits a rank $g-1$ compact factor by \cite[Cor. 5.3]{ForniDev} (summarized in \cite[Lem. 2.1]{AulicinoCompDegKZAIS}).  The affine submanifold $\cM$ has (cylinder) rank one by \cite[Cor. 4.2]{AulicinoCompDegKZAIS}.  It is arithmetic by \cite[Lem. 3.1]{AulicinoCompDegKZAIS}.  Therefore, $\cM$ is either an arithmetic Teichm\"uller curve with completely degenerate KZ-spectrum or it contains infinitely many Teichm\"uller curves with completely degenerate KZ-spectrum by \cite[Cor. 3.2]{AulicinoCompDegKZAIS}.  Since being a Teichm\"uller curve with completely degenerate KZ-spectrum is equivalent to being an ST-curve by \cite[Prop. 6.4]{MollerShimuraTeich} (Proposition \ref{EqToCDKZSpec} above) and there are only two such curves in any genus by Theorem \ref{FullSTClassThm}, the theorem follows.
\end{proof}


\bibliography{fullbibliotex}{}

\end{document}